\documentclass[11pt,letterpaper]{article}

\usepackage[affil-it]{authblk}
\usepackage[utf8]{inputenc}
\usepackage{amsmath}
\usepackage{amsfonts}
\usepackage{amssymb}
\usepackage{amsthm}
\usepackage{graphicx}
\usepackage{subcaption} 
\usepackage[left=2cm,right=2cm,top=2cm,bottom=2cm]{geometry}
\usepackage{enumerate}
\usepackage{appendix}
\usepackage{mathtools}
\usepackage{xcolor}
\usepackage{hyperref}
\usepackage{todonotes}
\usepackage{cleveref}
\usepackage{amsthm}
\theoremstyle{plain}

\newtheorem*{theorem*}{Theorem}

\newtheorem{theorem}{Theorem}[section]
\newtheorem{lemma}[theorem]{Lemma}
\newtheorem{prop}[theorem]{Proposition}
\newtheorem{corollary}[theorem]{Corollary}
\theoremstyle{definition}

\theoremstyle{remark}
\newtheorem{remark}[theorem]{Remark}

\DeclareMathOperator{\supp}{supp}

\newcommand{\ee}{\varepsilon}

\newcommand{\pare}[1]{\left( #1 \right)}
\newcommand{\set}[1]{\left\lbrace #1 \right\rbrace}

\newcommand{\norm}[1]{\left\| #1 \right\|}
\newcommand{\av}[1]{\left| #1 \right|}

\newcommand{\system}[1]{\left\{ #1 \right.}

\definecolor{darkgreen}{rgb}{0, 0.5, 0}

\usepackage{graphicx}

\newcommand\smallc{
  \mathchoice
    {{\scriptstyle\mathcal{C}}}
    {{\scriptstyle\mathcal{C}}}
    {{\scriptscriptstyle\mathcal{C}}}
    {\scalebox{.7}{$\scriptscriptstyle\mathcal{C}$}}
  }

\usepackage[maxbibnames=99,backend=bibtex,style=numeric-comp]{biblatex}

\numberwithin{equation}{section}

\date{}
\begin{document}

\title{\bf Quantitative aspects on the ill-posedness of the Prandtl and hyperbolic Prandtl equations}

\author{Francesco De Anna$\,^1$, Joshua Kortum$\,^2$,   Stefano Scrobogna$\,^3$}
\affil{\small
  $\,^1$ Institute of Mathematics, University of Würzburg, Germany\\
    email: francesco.deanna@uni-wuerzburg.de\\

    \vspace{0.3cm}
    $\,^2$ Institute of Mathematics, University of Würzburg, Germany \\
    email: joshua.kortum@uni-wuerzburg.de\\
    
     \vspace{0.3cm}
    $\,^3$ Dipartimento di Matematica e Geoscienze, Università degli Studi di Trieste, Italy, \\
    email: stefano.scrobogna@units.it
}

\maketitle

\begin{abstract}
 \noindent  
  We address a physically-meaningful extension of the Prandtl system, also known as hyperbolic Prandtl equations. We show that the linearised model around a non-monotonic shear flow is ill-posed in any Sobolev spaces. Indeed, shortly in time, we generate solutions that experience a dispersion relation of order $\sqrt[3]{k}$ in the frequencies of the tangential direction, akin the pioneering result of Gerard-Varet and Dormy in \cite{MR3925144} for Prandtl (where the dispersion was of order $\sqrt{k}$). We emphasise however that this growth rate does not imply (a-priori) ill-posedness in Gevrey-class $m$, with $m>3$ and we relate also these aspects to the original Prandtl equations in Gevrey-class $m$, with $m>2$.
  
  \noindent 
  By relaxing certain assumptions on the shear flow and on the solutions, namely allowing for unbounded flows in the vertical direction, we however determine a suitable shear flow for the mentioned ill-posedness in Gevrey spaces.
\end{abstract}

\section{Introduction}

\subsection{Presentation of the problem}

In this work we are mainly concerned with the following hyperbolic extension of the linearised Prandtl equations (cf.~\cite{PZ2022,DKS2023, MR4498949, aarach2023global, PZZ2020})
\begin{equation}
    \label{LHP}
    \system{
    \begin{alignedat}{4}
    & ( \partial_t +1)\big(
    \partial_t u  + U_{{\rm s}}  \partial_x u + v \ U'_{{\rm s}}\big) - \partial_{y}^2 u =0,
    \qquad (t,x,y)\in \,
    &&(0,T) \times \mathbb{T}\times \mathbb R_+,\\
    & \partial_x u + \partial_y v =0
    &&(0,T) \times \mathbb{T}\times \mathbb R_+,
    \\
    & \left. \pare{u, u_t}\right|_{t=0} =\pare{u_{\rm in}, u_{t,\rm in}}
    &&\hspace{1.425cm}\mathbb{T}\times \mathbb R_+,
    \\
    & \left. u\right|_{y=0}=0,\quad v_{|y = 0} = 0
     &&(0,T) \times \mathbb{T}.
    \end{alignedat}
    } 
\end{equation}
The main unknown is the scalar component $u = u(t,x,y)$ of the velocity field $(u, v)^T$, while $v = v(t,x,y)$ is formally determined by the divergence-free condition
\begin{equation*}
    v(t,x,y) = -\int_0^y \partial_x u(t,x,z)dz.
\end{equation*}
The function $U_s = U_s(y)$ is prescribed and defines a suitable shear flow $(U_s(y), 0)$ around which the model have been linearised. 
We aim to develop an ill-posedness theory of \eqref{LHP} in Sobolev spaces akin to the one in the seminal paper of G\'erard-Varet and Dormy \cite{GD2010}, for the linearization of the Prandtl equations around a shear flow. Our main result concerns a {\it norm inflation} phenomenon in Sobolev regularities: some initial data, that are of unitary size in $H^m$ along the variable $x\in \mathbb T$, generate solutions of inflated size $\delta^{-1}$ after a time $t = \delta > 0$ arbitrarily small (cf.~\Cref{thm:ill-posedness-Sobolev}).

\noindent 
In \cite{DKS2023}, we proved that System \eqref{LHP} is locally well-posed when the initial data have regularity Gevrey-class~3 along $x\in \mathbb T$. With this work, we aim to provide further insights on the well- or ill-posedness issue when the initial data are Gevrey-class $m$, $m>3$. We relate moreover these aspects to the original Prandtl equations in Gevrey-class $m$, $m>2$, supporting with certain remarks (cf.~\Cref{sec:remark-on-Gevrey-m>3}) the main result of G\'erard-Varet and Dormy in \cite{GD2010}.


\subsection{A brief overview on the analysis of the Prandtl equations}

System \eqref{LHP} shares several similarities with the classical Prandtl system,
whose leading equation is always on the state variable $u= u(t,x,y)$
\begin{equation*}
    \partial_t u + u\partial_x u + v \partial_y u - \partial_y^2 u = 0,\qquad 
    (t,x,y)\in 
    (0,T) \times \mathbb{T}\times \mathbb R_+
\end{equation*}
and  replaces the first equation in \eqref{LHP} (with initial data only on $u$). Nowadays, the Prandtl model belongs to the mathematical and physical folklore, predicting the dynamics of boundary layers in fluid mechanics. Numerous mathematical investigations have been performed during the past decades, in order to reveal the underlying instabilities of the solutions. In this paragraph, we shall briefly recall some of the results about the related  well- and ill-posedness issues. For further analytical problems like asymptotic behaviour, boundary-layer separations, boundary-layer expansions, homogenisation and more refined models, we refer for instance to \cite{DM2019,GLSS2009,PZ2021,MR2979511,MR4462474, MR3634071,Renardy2009,DG2022} and references therein.

\noindent 
Regarding the well-posedness of the classical Prandtl equations, there are two branches of assumptions that one might impose on the initial data: they are monotonic in the vertical variable (hence they do not allow for non-degenerate critical points) or they belong to suitable function spaces that control all derivatives along  the tangential variable (i.e.~analytic or Gevrey initial data). 

\noindent 
In case the initial data are monotonic, the system is known to be well-posed in standard functions spaces, such as Sobolev and H\"older. Roughly speaking this monotonicity precludes certain instabilities of the solutions (at least locally in time), such as the so-called boundary-layer separations. An analytic result in this direction was shown by Oleinik (see e.g.\ \cite{oleinik1963prandtl,Oleinik1999}) using the Crocco transform. More recently, the result was renovated by Masmoudi and Wong in \cite{MW2015}, and Alexandre, Wang, Xu and Yang in \cite{AWXY2015} without employing the latter. 

\noindent 
On the other hand, requiring the initial data to be highly regular suffices as well for the well-posedness of the Prandtl equations. Caflish and Sammartino \cite{SC1998} proved local-in-time well-posedness for initial data that are analytic in all directions. In reality, a refinement elaborated in \cite{lombardo2003well} showed that the analyticty condition is only required in the tangential variable. 

\noindent
Motivated by findings on ill-posedness, the set of functions spaces allowing for local existence theory was enlarged from analytic to the Gevrey scale. We mention here the works of G\'erard-Varet and Masmoudi \cite{GM2015} in the class $7/4$, Li and Yang \cite{LY2020} for small initial data in the class $2$, and Dietert and G\'erard-Varet \cite{MR3925144} achieving the well-posedness also for large data in the Gevrey class 2. \\
The latter regularity property is strongly suggested to be the borderline case. Indeed, smooth solutions might blow-up without analyticity or monotonicity assumptions as shown in \cite{weinan1997blowup}. But more drastically,  G\'erard-Varet and Dormy proved in  \cite{GD2010} the ill-posedness of the linearized Prandtl system in any Sobolev space by constructing solutions which experience a norm inflation as $e^{\sigma  \sqrt{k}t}$ (for short times, c.f.\ the discussion in Section \ref{sec:remark-on-Gevrey-m>3}), with $\sigma > 0$. These results are extended in \cite{Gerard-Varet-Nguyen-2012,guo2011note} by G\'erard-Varet and Nguyen, and Guo and Nguyen, to non-existence results of weak solutions for the linear Prandtl system in the energy space. The latter works additionally shed light on the fully non-linear system culminating in  proving Lipschitz ill-posedness. 

\noindent
In sum,  the well-posedness theory obtained in the literature reflects in general the physically observed instabilities related to the Prandtl equations. One of the questions which arise is if extended models as \eqref{LHP}  inhere better properties as a result of the finite speed propagation. In \cite{DKS2023} we provided a positive answer to this dilemma, enlarging the Gevrey scale from the class $2$ to the class $3$. With the current work we show however that System \eqref{LHP} is still not  desirable for the well-posedness in Sobolev spaces.

\noindent

\subsection{Some physical aspects of the model}
System \eqref{LHP} formally arises as a (linearised) model for the boundary layers of the Navier-Stokes equations, whose Cauchy stress tensor is ``delayed'' through a first-order Taylor expansion:
\begin{equation}\label{eq:delayed-Navier-Stokes}
    \partial_t U + U \cdot \nabla U + \nabla P = {\rm div}\,\mathbb S,\qquad 
    \mathbb{S}(t+ \tau, \cdot) \approx  \mathbb{S}(t, \cdot)+\tau \partial_t  \mathbb{S}(t, \cdot) =\mu \frac{\nabla U(t,\cdot)+\nabla U(t,\cdot)^T}{2}. 
\end{equation} 
We refer to \cite{ADPZ2022,DKS2023} and references therein for a deeper discussion on the physical motivations and derivation of the model. We point out that this delayed relation on $\mathbb S$ was introduced in fluid-dynamics by Carrassi and Morro~\cite{CARRASSIMORRO},  inspired by the celebrated work of Cattaneo \cite{Cattaneo1949, Cattaneo1958} on heat diffusion. Moreover, System \eqref{eq:delayed-Navier-Stokes} represents somehow a simplified version of the celebrated Oldroyd-B model (or Maxwell equations, cf.~for instance \cite{chun}):
\begin{equation*}
    \partial_t U + U \cdot \nabla U - \nu \Delta u + \nabla P  = \mu_1 {\rm div}\, \mathcal{T},\quad 
    \partial_t \mathcal{T} + u\cdot \nabla  \mathcal{T} + \gamma  \mathcal{T} - \nabla U \mathcal{T}   - (\nabla U)^T  \mathcal{T} = \mu_2 
    \frac{\nabla U(t,\cdot)+\nabla U(t,\cdot)^T}{2},
\end{equation*}
where the constant $\gamma $ represents the time scale for the elastic relaxation. Equations \eqref{eq:delayed-Navier-Stokes} neglects nevertheless several terms of Oldroyd-B, in particular the convective and upper convective derivatives. We might therefore interpret our results as precursors for the analytical understanding of boundary layers of viscoelastic fluids.

\subsection{Main Result}
For the sake of comparison, we employ a similar Ansatz as the one of G\'erard-Varet and Dormy in~\cite{GD2010}. All along the present manuscript, we denote by $W^{s, \infty}_\alpha$, with $s\geq 0$ and $\alpha>0$, the following weighted Sobolev spaces in $\mathbb R_+ = [0, \infty)$
\begin{align*}
    W^{s, \infty}_\alpha 
    = 
    W^{s, \infty}_\alpha 
    (\mathbb R_+)
    := & \ \set{ f\in W^{s, \infty}(\mathbb R_+)  
    \; 
    \text{such that} 
    \; 
    y \in \mathbb R_+ \to e^{\alpha y}f(y)
    \; 
    \text{is also in } 
    \; 
    W^{s, \infty}(\mathbb R_+)},    
\end{align*}
endowed with the norm $ \norm{f}_{W^{s,\infty}_\alpha} := \norm{e^{\alpha \cdot } f }_{ W^{s, \infty}}$. For the entirety of this manuscript $\alpha>0$ is fixed and might be imposed equal to $1$ (we keep the notation $W^{s,\infty}_\alpha$, mainly because of \cite{GD2010}). 
Since our solutions are periodic in the tangential direction $x\in \mathbb T$, we also set  $H^{m} W^{s, \infty}_{\alpha} =  H^m(\mathbb T,W^{s,\infty}_{\alpha}(\mathbb R_+))$, $m\geq 0$ and also 
$\mathcal{C}^\omega  W^{0, \infty}_{\alpha} = \mathcal{C}^\omega (\mathbb T,W^{0, \infty}_{\alpha}(\mathbb R_+))$
, making use of the Fourier transform in $x\in \mathbb T$:
\begin{equation*}
\begin{aligned}
    H^{m} W^{0, \infty}_{\alpha}
    &=
    \bigg\{
        f =  \sum_{k\in \mathbb Z}f_k(y)e^{ikx} \quad (x,y)\in \mathbb T\times \mathbb R_+,
        \quad 
        \| f \|_{H^{m} W^{s, \infty}_{\alpha}}
        :=
        \bigg(
        \sum_{k\in \mathbb Z}
        (1+k^2)^m
        \| f_k \|_{W^{s, \infty}_{\alpha}}^2
        \bigg)^\frac{1}{2}
        < \infty
    \bigg\},\\
    \mathcal{C}^\omega  W^{0, \infty}_{\alpha}
    &=
    \bigg\{
        f =  \sum_{k\in \mathbb Z}f_k(y)e^{ikx} \quad (x,y)\in \mathbb T\times \mathbb R_+,
        \quad 
        \sup_{k\in \mathbb Z}
        e^{\sigma |k|}
        \| f_k \|_{W^{s, \infty}_{\alpha}}
        < \infty,
        \quad
        \text{for a given}
        \quad 
        \sigma>0
    \bigg\}.
\end{aligned}
\end{equation*}
Our main result expresses the ill-posedness of System \eqref{LHP} in $H^{m} W^{0, \infty}_{\alpha}$ when the shear flow $U_s(y)$ is non-monotonic. It develops around a family of solutions that are of unitary size in Sobolev spaces  at $t= 0$ and experience an inflation of the norm after any short time $t = \delta >0$.
\begin{theorem}\label{thm:ill-posedness-Sobolev}
    Assume that the shear flow $U_s = U_s(y)$ is in $W^{4,\infty}_\alpha(0,\infty)$ and satisfies the structural relations $U_s(a) = U_s'(a) = 0$ together with $U_s''(a) \neq  0$, for a given $a\in (0, \infty)$. 
    Then for all $m\geq 0$, $\mu \in [0, 1/3)$ and small time $\delta>0$, there exists a pair of initial data
    \begin{equation*}
    \begin{aligned}
        \pare{u_{\rm in}, u_{t,\rm in}} \in 
        \mathcal{C}^\omega
        (\mathbb T
        &,W^{1,\infty}_{\alpha}(\mathbb R_+))
        \times 
        \mathcal{C}^\omega
        (\mathbb T,W^{0, \infty}_{\alpha}(\mathbb R_+))
        \hookrightarrow
        H^m
        (\mathbb T,W^{1,\infty}_{\alpha}(\mathbb R_+))
        \times 
        H^m
        (\mathbb T,W^{0, \infty}_{\alpha}(\mathbb R_+))
        \\
        &\text{with} 
        \quad 
        \| u_{\rm in} 
        \|_{H^m W^{1, \infty}_{\alpha}}
        +
        \| u_{t,\rm in} 
        \|_{H^m W^{0, \infty}_{\alpha}}
        \leq 1,  
    \end{aligned}
    \end{equation*}
    that generates a global-in-time smooth solution $u$ of \eqref{LHP}, which satisfies
    \begin{equation}\label{Sobolev-norm-inflation-delta}
        \sup_{0\leq t \leq \delta }
        \| u(t) \|_{H^{m-\mu} W^{0, \infty}_{\alpha}}
        \geq \frac{1}{\delta}.
    \end{equation}
\end{theorem}

\begin{remark}
    We allow the solutions $u$ in \Cref{thm:ill-posedness-Sobolev} and the suitable family of initial data to take values in $\mathbb C$. Nevertheless, since the shear flow as well as all other coefficients \eqref{LHP} are real valued, taking the real (or imaginary part respectively) provide (non-trivial) real-valued solutions which undergo the same instability mechanism.

    \noindent 
    The parameter $\mu \in [0, 1/3)$ allows to enlarge the Sobolev space for the values of the solutions. As long as the regularity index of $H^{m-\mu}$ differs less than $1/3$  from the original space $H^m$ (a reminiscent of the Gevrey-class $3$ regularity), the problem is still ill-posed. An analogous threshold was shown in \cite{GD2010} with $\mu \in [0, 1/2)$ (reminiscent of Gevrey-class $2$).
\end{remark}

\noindent 
As for its homologous problem in \cite{GD2010}, the instability process described by \Cref{thm:ill-posedness-Sobolev} is mainly enabled by a meaningful family of initial data, which highly oscillate in the tangential variable $x\in \mathbb T$. More precisely, at high frequencies $k\gg 1$, we select initial data $\pare{u_{\rm in}, u_{t,\rm in}}$ on the eigenspace of $e^{ik x}$ and determine a suitable non-trivial asymptotic of the $(t,y)$-dependent Fourier coefficients as frequencies $k\to \infty$ (more details are provided starting from \Cref{sec:outline-of-the-proof}). We show that the corresponding solutions grow as $e^{\sigma \sqrt[3]{k}t}$ at least for a very short time depending on the frequencies (cf.~\Cref{sec:remark-on-Gevrey-m>3},  \Cref{prop:main-prop} and \Cref{lemma:inflation-of-Tk}). 
Moreover, as $k\to \infty$, the profiles in $y\in \mathbb R_+$ of the initial data relate to a  suitable ``spectral condition'' for the following ordinary differential equation:
\begin{lemma}\label{lemma:W}
There exists a complex number $\gamma \in \mathbb C$ with ${\rm Im }(\gamma)<0$ and a complex solution $W : \mathbb R \to \mathbb C$ of the ordinary differential equation
\begin{equation}\label{eq:equation-of-W-tilde-intro}
    \gamma
    \big( 
        \gamma - z^2
    \big)^2
    \frac{d}{d z} W(z)
    +
    \frac{d^3}{d z^3}
    \bigg[ 
        \big(
           \gamma - z^2
        \big)W(z) 
    \bigg]= 0,\qquad z\in \mathbb R,
\end{equation}
such that $\lim_{z\to -\infty} W(z) = 0$ and $\lim_{z\to +\infty} W(z) = 1$.     
\end{lemma}
\noindent 
The ODE \eqref{eq:equation-of-W-tilde-intro} is similar to its homologous $(1.7)$ in \cite{GD2010} (cf.~also here \eqref{eq:GV-Dormy-equation-of-W}). They do present however some technical differences, that are mainly due to the fact that the leading operator of Prandtl is a heat equation,  whereas the leading term in \eqref{LHP} is a wave equation (at high frequencies $k \gg 1$).





\subsection{About the ill-posedness of the Prandtl equation}\label{sec:remark-on-Gevrey-m>3}
The pioneering work of  Dietert and G\'erard-Varet in \cite{MR3925144} established the local-in-time well-posedness of the Prandtl equations, when the initial data have Gevrey-class-2 regularity along the tangential direction $x\in \mathbb T$. Before their result,  the exponent of Gevrey-class 2 was attained only in the special setting prescribed by \cite{LY2020}, in light also of the previous investigation of G\'erard-Varet and Dormy \cite{GD2010} on the ill-posedness of Prandtl in Sobolev spaces. Indeed, for any sufficiently-large frequency $k\gg 1$, G\'erard-Varet and Dormy succeeded in establishing 
solutions of the linearised Prandtl equations around a non-monotonic shear flow $(U_s(y), 0)$
\begin{equation}\label{eq:linearised-Prandtl-autonomous}
    \partial_t u + U_s(y) \partial_x u + v U_s'(y) - \partial_y^2 u  = 0,
    \qquad 
    \partial_x u + \partial_y v = 0,
    \qquad (t,x,y)\in (0,T) \times \mathbb T \times \mathbb R_+,
\end{equation}
experiencing a dispersion relation proportional to $\sqrt{k}$. Their result was later on strengthened by further investigations on related problems, we report for instance G\'erard-Varet and Nguyen in \cite{Gerard-Varet-Nguyen-2012} and Ghoul, Ibrahim, Lin and Titi in~\cite{Ghoul-Ibrahim-Lin-Titi-2022}, emphasising that there exist solutions experiencing growth of order $e^{\sqrt{k}t}$ in the frequencies:
\begin{itemize}
    \item { \it  The authors  $\dots$ construct $\mathcal{O}(k^{-\infty})$ approximate solutions that grow like $e^{\sqrt{k}t} $ for high frequencies $k$ in $x$},
    \item 
    {\it its linearisation around a special background flow has unstable solutions of similar form, but with $\sigma_k \sim \lambda \sqrt{k}$, for $ k \gg 1$ arbitrarily large and some positive $\lambda \in \mathbb R_+$.}
\end{itemize}
Because of this growth rate, one may wonder why  G\'erard-Varet and Dormy addressed the ill-posedness in Sobolev spaces rather than in any Gevrey-class $m$, with $m>2$. With this section, we want to highlight some quantitative aspects related to the proof in  \cite{GD2010}  that play a fundamental role in the present manuscript as well. Furthermore, in our case we build solutions of the hyperbolic extension~\eqref{LHP} with a dispersive relation of order $\sqrt[3]{k}$, thus these remarks reveal also our choice of Sobolev rather than Gevrey-class $m$, with $m>3$. 

\noindent 
Roughly speaking, the most compelling reason resides in the maximal lifespan for which the mentioned growths occur: it is proportional to $t \sim \ln(k)/\sqrt{k}$ for the Prandtl equation and to $t\sim   \ln(k)/\sqrt[3]{k}$ for the hyperbolic extension. To be more specific, we state the following corollary of Theorem~1 in  \cite{GD2010}, when applied to the autonomous system \eqref{eq:linearised-Prandtl-autonomous}.
\begin{corollary}[due to Theorem~1 in \cite{GD2010}]\label{cor:GV-Dormy}
    Let $U_s\in W^{4, \infty}_\alpha(\mathbb R_+)$ and assume that $U_s'(a) = 0$ and $U_s''(a) \neq 0$ for some $a>0$. Denote by $T$ the semigroup of \eqref{eq:linearised-Prandtl-autonomous} for analytic solutions (cf.~Proposition~1 in  \cite{GD2010}). There exists $\sigma>0$, such that  
    \begin{equation*}
        \sup_{0\leq t \leq \delta }
        e^{-\sigma t \sqrt{\partial_x}}
        \big\| T(t) \big\|_{\mathcal{L}(H^mW^{0, \infty}_\alpha, H^{m-\mu}
        W^{0, \infty}_\alpha)} = + \infty,
    \end{equation*}
    for all $\delta >0$, $m\geq 0$ and $\mu \in [0, 1/2)$.
\end{corollary}
\noindent 
The result of G\'erard-Varet and Dormy implies in particular that $T$ does not extend to an operator in the Sobolev space $H^mW^{0, \infty}_\alpha$. 
Following the proof in \cite{GD2010}, one can obtain in reality a more refined version of \Cref{cor:GV-Dormy}, which provides indeed an explicit lower bound of the semigroup $T_k$ (the projection of $T$ on the eigenspace generated by $e^{ikx}$), associated to
\begin{equation}\label{eq:linearised-Prandtl-autonomous-k}
    \partial_t u_k + ik U_s(y)  u_k + v_k U_s'(y) - \partial_y^2 u_k  = 0,
    \qquad 
    ik  u_k + \partial_y v = 0,
    \qquad (t,x,y)\in (0,T) \times \mathbb R_+,
\end{equation}
with $v_k = u_k = 0$ in $y = 0$, and an explicit upper bound for the time in which the instability occurs. 
\begin{theorem}\label{thm:refined-verion-GV-Dormy}
    Let $U_s\in W^{4, \infty}_\alpha(\mathbb R_+)$ and assume that $U_s'(a) = 0$ and $U_s''(a) \neq 0$ for some $a>0$. Then there exists $\sigma_0>0$ such that for all  $\sigma \in [0, \sigma_0)$ and $k\in \mathbb N$, the following inequality 
    \begin{equation}\label{eq:refined-verion-GV-Dormy}
        \sup_{0\leq t \leq \frac{1}{2(\sigma_0-\sigma)}
        \frac{\ln(k)}{\sqrt{k}}
        }
        e^{-\sigma t k^\frac{1}{2}}
        \big\|  T_k(t) 
        \big\|_{\mathcal{L}(W^{0, \infty}_\alpha ) }
        >
        \mathcal{C}_{\sigma}
        k^\frac{1}{2},
    \end{equation}
   holds true for a constant $\mathcal{C}_{\sigma}$, which depends uniquely upon $\sigma$ and $\sigma_0$.
\end{theorem}
    \noindent 
    \Cref{thm:refined-verion-GV-Dormy} implies in particular \Cref{cor:GV-Dormy}. Indeed, at any small time $t= \delta>0$, we may consider a sufficiently large frequency ${\bf k}\gg 1$ with $\ln({\bf k})/\sqrt{{\bf k}}\leq 2(\sigma-\sigma_0) \delta$, so that
    \begin{equation*}
    \begin{aligned}
        \sup_{0\leq t \leq \delta }
        e^{-\sigma t \sqrt{\partial_x}}
        \big\| T(t) \big\|_{\mathcal{L}(H^mW^{0, \infty}_\alpha, H^{m-\mu}
        W^{0, \infty}_\alpha)}
        &= 
        \sup_{0\leq t \leq \delta }
        \sup_{k \in \mathbb Z}
        k^{-\mu }
        e^{-\sigma t k^\frac{1}{2}}
        \big\|  T_k(t) 
        \big\|_{\mathcal{L}\big(
        W^{0, \infty}_\alpha,  W^{0, \infty}_\alpha  \big) }\\
        & 
        \geq
        \sup_{0\leq t \leq \delta }
        \sup_{k \geq  {\bf k}}
        k^{\frac 12 -\mu }
        e^{-\sigma t k^\frac{1}{2}}
        = + \infty.
    \end{aligned} 
    \end{equation*}
    Remarkably, \Cref{thm:refined-verion-GV-Dormy} provides an explicit inflation of the norms at any frequency $k\in \mathbb N$, as well as a related maximum lifespan. This time is proportional to $t \sim \ln(k)/\sqrt{k}$ and was somehow already observed in the proof of \cite{GD2010} (cf.~pag.~602 in \cite{GD2010}, where the authors obtained a contradiction considering a time $t \gg  \frac{\mu}{\sigma-\sigma_0}| \ln(\ee)|\sqrt{\ee}$, with $\ee = 1/k$). 

    \noindent 
    On a more technical level, the proof in \cite{GD2010} develops around the following ansatz for a velocity field $u(t,x,y) =  e^{i k (-U_s(a) + \tau k^{-1/2})t + ik x}\hat{u}_k(y)$, where $\tau\in \mathbb C$ and
    $\sigma_0:={\mathcal{I}m(\tau)}<0$. This function $u$ experiences therefore a-priori a growth as $e^{\sigma_0 \sqrt{k}t}$, at any time $t>0$.     
    However, we shall emphasise that this flow is a solution $u^{\rm fr}= u^{\rm fr}_k(t,y)e^{ik x}$ of a ``forced'' version of the Prandtl equations,
    which depends on a non-trivial remainder $r_k = \mathcal{R}_k(t,y)e^{ikx}$ (cf.~(4.2) in \cite{GD2010}, with  $\ee = 1/k$ and $u_\ee=u^{\rm fr}$). 
    Hence, in order to transfer the instability of $u^{\rm fr}$ to a homogeneous solution $u$ of Prandtl, the authors invoked the Duhamel's identity
    \begin{equation*}
        u^{\rm fr}_k(t,y) - u_k(t,y) = \int_0^t T_k(t-s)(\mathcal{R}_k(s))(y) ds,
    \end{equation*}
    (cf.~the identity below (4.2) in \cite{GD2010}, where $\tilde U_\ee = u^{\rm fr}_k$ and  $U_\ee= u_k$). Roughly speaking, $u_k$ behaves similarly as $u^{\rm fr}_k$, as long as the integral at the r.h.s~of Duhamel is sufficiently small. \Cref{thm:refined-verion-GV-Dormy} translates this smallness relation in terms of the semigroup $T_k$ and the maximal lifespan $t  = \frac{1}{2(\sigma_0-\sigma)} \frac{\ln(k)}{\sqrt{k}}$. We obtain a similar result also to the hyperbolic extension \eqref{LHP} (cf.~\Cref{sec:outline-of-the-proof} and \Cref{lemma:inflation-of-Tk}).

\begin{remark}
    Finally, we shall remark that the norm inflation of \Cref{thm:refined-verion-GV-Dormy} is inefficient in providing ill-posedness in $\mathcal{G}_\sigma^mW^{0, \infty}_\alpha$ (Gevrey-class $m$ in $x\in \mathbb T$), with $m>2$. The analytic semigroup $T(t)$ never extends to an operator from $\mathcal{G}_\sigma^mW^{0, \infty}_\alpha$ into $\mathcal{G}_\eta^mW^{0, \infty}_\alpha$ (with a smaller radius of regularity $\eta\in [0,\sigma)$) if and only if
    \begin{equation}\label{est:ill-posedness-in-Gevrey}
        \sup_{k \in \mathbb Z}
        \sup_{0\leq t \leq \delta } e^{(\eta-\sigma) k^{\frac{1}{m}} } \| T_k(t) \|_{\mathcal{L}(W^{0, \infty}_\alpha)} = + \infty,
    \end{equation}
    for any $ \delta>0$. Inequality \eqref{eq:refined-verion-GV-Dormy} does not automatically imply \eqref{est:ill-posedness-in-Gevrey}, since it would require that \eqref{eq:refined-verion-GV-Dormy} holds true at a time $t \sim \frac{\sigma-\eta}{\sigma} k^{\frac 1m - \frac{1}{2}}$, which satisfies  $e^{(\eta-\sigma) k^{\frac 1m } + \sigma t k^{ \frac 12}}\sim 1$. The upper bound $ t \leq  \frac{1}{2(\sigma_0-\sigma)} \frac{\ln(k)}{\sqrt{k}}$ is hence too restrictive for the ill-posedness of the  linearised Prandtl equations in any Gevrey-class $m$, with $m>2$. 
\end{remark}

\noindent 
On the other hand, if we allow $U_s(y)$ to be unbounded in $y\in [0,\infty)$, we may consider a specific shear flow of the form $U_s(y) = (y-a)^2/2$, which satisfies indeed the assumptions of \Cref{thm:refined-verion-GV-Dormy} at $y = a$ (but does not decay exponentially as $y\to \infty$). In this case, the forced solution $u^{\rm fr}_k$ (derived in the proof of \cite{GD2010}) corresponds in reality to an exact solution of the Prandtl equation: $u^{\rm fr}_k = u_k$ (the forcing term $\mathcal{R}_k$ vanishes at any time for this specific choice of $U_s$, cf.~\Cref{remark:it-might-be-true}). In this case, the growth $e^{\sigma_0 \sqrt{k}t}$ holds true at any time $t>0$, which therefore implies ill-posedness in any Gevrey-class $m$, with $m>2$, along the tangential direction $x\in \mathbb T$. However, also in this case there are some drawbacks: the built solution $u_k $ is unbounded and diverge to $\infty$  as $y\to \infty$. This increase in $y\in [0,\infty)$ is a reminiscent of the analogous behaviour of $U_s(y) = (y-a)^2/2$, making the Prandtl solution somehow unphiysical. Nevertheless, it might suggest that the time barrier $t\sim \ln(k)/\sqrt{k}$ is (momentarily) only mathematical and that the dispersion rate might hold true at any time $t>0$ also when considering non-monotonic shear flows $U_s(y)$ that decade as $y\to \infty$. 


\noindent 
The main goal of this paper is to address the ill-posedness of the hyperbolic extension \eqref{LHP} in Sobolev spaces. We however provide a short proof of \Cref{thm:refined-verion-GV-Dormy}
, as further support of the pioneeric work of G\'erard-Varet and Dormy in \cite{GD2010} (cf.~\Cref{sec:proof-of-GV-theorem}).

\section{Outline of the proof}\label{sec:outline-of-the-proof}
In this section, we illustrate the general principles that we set as basis of our proof, and we postpone the technical details to the remaining paragraphs. Our proof develops along three major axes:
\begin{enumerate}[(i)]
    \item We project the main equation \eqref{LHP} into frequency eigenspaces in $x\in \mathbb T$ (cf.~\eqref{LHP-k-s}) and we perturb the resulting equations with some meaningful forcing terms (cf.~\eqref{LHP-k} and \Cref{prop:main-prop}). This ensures a specific exponential growth (typical of Gevrey-class 3 regularities) on certain inhomogeneous solutions. In particular, this exponential growth holds true, globally in time.
    \item Locally in time, we transpose the growth described in (i) to the original homogeneous equation. Contrarily to (i), this holds uniquely for a very short time, which depends in particular upon the frequency of each eigenspace (cf.~\Cref{lemma:inflation-of-Tk}).
    \item  Finally, making use of the homogeneous solutions built in (ii), we derive an inflation of the Sobolev norms of the solutions of \eqref{LHP} as described by \Cref{thm:ill-posedness-Sobolev}.
\end{enumerate}
We postpone the major aspects of part (i) to \Cref{sec:proof-of-proposition-inhomogeneous} and we devote this section to some underlying remarks and the details of parts (ii) and (iii). 

\subsection*{(i) The inhomogeneous equation and the global-in-time growth of Gevrey-3 type}
We first formally project the main equation \eqref{LHP} to the eigenspace of each positive frequency $k\in \mathbb N$ in the $x$-variable, by means of the Ansatz 
\begin{equation*}
    u(t,x,y) = e^{ik x} u_k(t,y),\quad v(t,x,y) = k e^{ik x} v_k(t,y),\quad 
    (t,x,y)\in (0,T) \times \mathbb T \times \mathbb R_+.
\end{equation*}
We hence look for a suitable family of $(t,y)$-dependent profiles $u_k$, 
which solve the homogeneous system
\begin{equation}
    \label{LHP-k-s}
    \system{
    \begin{alignedat}{4}
    & ( \partial_t +1)\big(
    \partial_t u_k  + ik U_{{\rm s}} u_k +  k v_k U'_{{\rm s}}\big) - \partial_{y}^2 u_k =0
    \qquad (t,y)\in \,
    &&(0,T) \times \mathbb R_+,\\
    &\, i  u_k + \partial_y v_k =0
    &&(0,T) \times \mathbb R_+,
    \\
    & \left. \pare{ u_k, \partial_t u_{k}}\right|_{t=s} =\pare{u_{k,\rm in}, u_{k,t,\rm in}}
    &&\hspace{1.425cm}  \mathbb R_+,
    \\
    & \left.  u_k\right|_{y=0}=0,\quad  \left.  v_k \right|_{y = 0} = 0
     &&(0,T).
    \end{alignedat}
    } 
\end{equation}
Since this equation is (roughly) a linear damped wave equation at each fixed frequency $k\in \mathbb N$, we infer that any initial data $(u_{k,\rm in}, u_{k,t,\rm in})\in W^{1, \infty}_\alpha \times W^{0, \infty}_\alpha$ generates a unique global-in-time weak solution $u_k\in L^\infty(0,T;W^{0, \infty}_\alpha)$, for any lifespan $T>0$. 
Furthermore, we also infer that $u_k$ can be written in terms of a semigroup $T_k$
\begin{equation*}
   u_k(t) = T_k(t)\pare{u_{k,\rm in}, u_{k,t,\rm in}},
   \quad 
   \text{with}
   \quad 
   T_k(t) 
   :W^{1, \infty}_\alpha \times W^{0, \infty}_\alpha
   \to W^{0, \infty}_\alpha,
\end{equation*}
for any $T>0$ and any $t\in (0,T)$. Certainly, $u_k$ satisfies additional regularities, however  our central goal is to estimate $u_k$ in $ L^\infty(0,T;W^{0, \infty}_\alpha)$ and to determine a suitable growth of $\|u_k(t) \|_{W^{0, \infty}_\alpha}$ as time increases. 


\noindent
As  depicted in part $(i)$, we momentarily allow for perturbation of equation \eqref{LHP-k-s} by means of a general forcing term $f_k\in L^\infty(0,T;W^{0, \infty}_\alpha)$. To avoid confusion in the notation, we set this inhomogeneous version in the  state variable $u_k^{\rm fr}=u_k^{\rm fr}(t,y)$, which reflects a ``forced'' version of the hyperbolic Prandtl equation: 
\begin{equation}
    \label{LHP-k}
    \system{
    \begin{alignedat}{4}
    & ( \partial_t +1)\big(
    \partial_t u_k^{\rm fr}  + ik U_{{\rm s}} u_k^{\rm fr} +  k v_k^{\rm fr} U'_{{\rm s}}\big) - \partial_{y}^2 u_k^{\rm fr} =f_k,
    \qquad (t,y)\in \,
    &&(0,T) \times \mathbb R_+,\\
    &\, i u_k^{\rm fr} + \partial_y v_k^{\rm fr} =0
    &&(0,T) \times \mathbb R_+,
    \\
    & (u_k^{\rm fr}, \partial_t u_{k}^{\rm fr})|_{t=0} =\pare{u_{k,\rm in}, u_{k,t,\rm in}}
    &&\hspace{1.425cm}  \mathbb R_+,
    \\
    & \left. u_k\right|_{y=0}=0,\quad  \left. v_k \right|_{y = 0} = 0
     &&(0,T).
    \end{alignedat}
    } 
\end{equation}
We hence aim to determine a meaningful non-trivial forcing term $f_k$, such that \eqref{LHP-k} admits a solution $u_k^{\rm fr}$ of \eqref{LHP-k}, whose norm $\| u_k^{\rm fr}(t) \|_{W^{0,\infty}_\alpha}$ experiences an exponential growth in time as $e^{\sigma_0 t k^{1/3}}$, for a suitable constant $\sigma_0>0$. This growth holds true at any time $t \in (0,T)$, as described by the following proposition.

\begin{prop}\label{prop:main-prop}
Assume that the shear flow $U_s= U_s(y)$ is in $W^{4,\infty}_\alpha(0,\infty)$ and satisfies the relations $U_s(a) = U_s'(a) = 0$ together with $U_s''(a) \neq  0$, for a given $a\in (0, \infty)$. For any positive frequency $k\in \mathbb N$, there exist two non-trivial profiles $\mathbb U_k,\,\mathcal{R}_k :\mathbb R_+ \to \mathbb C$ in $W^{2, \infty}_\alpha (\mathbb R_+)$ and $W^{0, \infty}_\alpha (\mathbb R_+)$, respectively, and there exists a complex number $\tau \in \mathbb C$ only depending on $U_s$ with negative imaginary part ${\mathcal{I}m(\tau)}<0$, such that the initial data and the forcing term
\begin{equation}\label{ansatz:main-proposition}
     u_{k,\rm in}(y) =  \mathbb{U}_k(y),\quad 
     u_{k,t,\rm in}(y) =  i\tau k^\frac{1}{3}\mathbb{U}_k(y),\quad 
     f_k(t,y)  = -k^\frac{4}{3}e^{i\tau k^{\frac{1}{3}}t}\mathcal{R}_k(y),
     \quad 
     (t,y)\in (0,T)\times \mathbb R_+,
\end{equation}
generate a global-in-time solution $u_k^{\rm fr}\in L^\infty(0,T;W^{0, \infty}_\alpha)$ of \eqref{LHP-k}, which can be  written explicitly in the following form
\begin{equation}\label{eq:explicit-form-ukfr}
\begin{aligned}
    u_k^{\rm fr}(t,y) &= e^{i \tau k^{\frac{1}{3}}t} \mathbb{U}_k(y),\quad 
    v_k^{\rm fr}(t,y) = k e^{i \tau k^{\frac{1}{3}}t} \mathbb{V}_k(y),
    \quad\text{with}\quad  
    \mathbb{V}_k(y) :=-i \int_0^y\mathbb{U}_k(z)dz.
\end{aligned}
\end{equation}
The sequences  $(\mathbb U_k )_{k\in \mathbb N} $ is in $W^{2, \infty}_\alpha (\mathbb R_+)$ and it is uniformly bounded from above and below in $W^{0, \infty}_\alpha (\mathbb R_+)$, namely there exist two constants $\smallc,\mathcal{C}>0$ such that
\begin{equation}\label{uniform-estimate-for-Uk}
0<\smallc \leq 
    \| \mathbb{U}_k \|_{W^{0, \infty}_\alpha(\mathbb R_+)}
\leq \mathcal{C}<\infty.
\end{equation}
Furthermore, at high frequencies, the sequence $( k^\frac{4}{3}\mathcal R_k)_{k\in \mathbb N} $ is uniformly bounded in $W^{0, \infty}_\alpha (\mathbb R_+)$: there exists a constant $\mathcal{C}_\mathcal{R}>0$, which depends uniquely on the shear flow $U_s$, such that
\begin{equation}\label{uniform-bound-on-Rk}
    k^\frac{4}{3}\| \mathcal{R}_k \|_{W^{0, \infty}_\alpha} 
    \leq 
    \mathcal{C}_\mathcal{R}<\infty,\qquad \text{for any }k\in \mathbb N.
\end{equation}
\end{prop}
\begin{remark}
   Denoting by $\sigma_0:=-{\mathcal{I}m(\tau)}>0 $, \Cref{prop:main-prop} implies that $\| u^{\rm fr}_k(t) \|_{W^{0, \infty}_\alpha}\sim e^{\sigma_0 k^{1/3}t}$, at any $t\in (0,T)$, thanks to the explicit form \eqref{eq:explicit-form-ukfr} of the forced solution $u_k^{\rm fr}$. Indeed, invoking also the uniform estimates from below in \eqref{uniform-estimate-for-Uk}, we gather
    \begin{equation}\label{uniform-bound-of-ukfr}
        \| u_k^{\rm fr}(t) \|_{W_\alpha^{0, \infty}(\mathbb R_+)} = e^{-{\mathcal{I}m(\tau)} k^{\frac{1}{3}}t} \| \mathbb U_k  \|_{ W_\alpha^{0, \infty}(\mathbb R_+)}
        \geq 
        \smallc 
        e^{\sigma_0 k^{\frac{1}{3}}t},\qquad 
        \text{for any }t \in (0,T).
    \end{equation}
    Unfortunately, this method of determining an explicit solution as in \eqref{eq:explicit-form-ukfr} seems to work uniquely for the inhomogeneous system \eqref{LHP-k} and it is somehow inefficient with its homogeneous counterpart \eqref{LHP} (thus with the original system).    
    This is mainly due to the sequence of remainders $(\mathcal R_k)_{k\in \mathbb N}$, which define non-trivial forces $(f_k)_{k\in \mathbb N}$ in \eqref{eq:explicit-form-ukfr}. From \eqref{eq:explicit-form-ukfr} and \eqref{uniform-estimate-for-Uk}, 
    $\| f_k(t) \|_{W^{0, \infty}_\alpha}$ does not vanish (a-priori) as $k\to \infty$. We will determine indeed an explicit form of $\mathcal R_k$ (cf.~\eqref{eq:explicit-Rk}), which implies in particular
    \begin{equation*}
        \lim_{k\to \infty }k^\frac{4}{3}\mathcal{R}_k(y)  = i H(y -a) U_s'''(y)\neq 0,
        \quad \text{for all}\quad y>0\quad \text{with}\quad y \neq a,
    \end{equation*}
    where $H$ stands for the Heaviside step function. The contribution of $f_k$ to the global-in-time instability seems to be not negligible.  
    However, we show in $(ii)$ that we can still obtain a similar result with the homogeneous equation, at least for a very short time that depends on the frequencies.
\end{remark}
\noindent 
Since the proof of \Cref{prop:main-prop} is rather technical, we postpone it to \Cref{sec:proof-of-proposition-inhomogeneous}. We anticipate that it follows a similar approach as the one  used in \cite{GD2010}: we plug the ansatz \eqref{ansatz:main-proposition} to the main equations \eqref{LHP-k}, we analyse the asymptotic limit as frequencies $k \to \infty$ and we reduce the problem to a ``spectral condition'' on a related ODE. 
We shall hence devote the remaining parts of this section to address the details of $(ii)$ and $(iii)$ and thus to the proof of \Cref{thm:ill-posedness-Sobolev}. 

\subsection*{(ii) An instability of the homogeneous equation at a short time $t\sim \ln(k)/\sqrt[3]{k}$}

Based on the result given by \Cref{prop:main-prop}, we aim to obtain a similar instability for the homogeneous system \eqref{LHP-k-s}, which will lead in $(iii)$ to the ill-posedness of \eqref{LHP} in Sobolev spaces. 
We first invoke the following Duhamel's formula, which relates a general forced solution $u_k^{\rm fr}$ of \eqref{LHP-k} with the semigroup $T_k$ of the homogeneous system \eqref{LHP-k-s}:
\begin{equation}\label{Duhamel-formula}
    u_k^{\rm fr}(t,y) = 
    \underbrace{T_k(t)(u_{k, \rm in}, u_{k,t,\rm in})(y)}_{=:u_k(t,y)}
    + 
    \int_0^t
     T_k(t-s)(0, f_k(s))(y)ds.
\end{equation}
Here we shall interpret $u_k$ as the unique solution of the homogeneous problem \eqref{LHP-k-s}, with same initial data of $u_k^{\rm fr}$. Roughly speaking, 
in order to transfer the instability of $u_k^{\rm fr}$ to $u_k$, we need to ensure that the integral on the r.h.s.~of \eqref{Duhamel-formula} remains sufficiently small. With the following lemma, we translate  this condition directly as a property of the semigroup $T_k(t)$ at a time $t$, which is (at maximum) proportional to $\ln(k)/\sqrt[3]{k}$ (thus a time that vanishes as the frequency $k\to \infty)$.  
\begin{lemma}\label{lemma:inflation-of-Tk}
    Assume that the shear flow $U_s= U_s(y)$ is in $W^{3,\infty}_\alpha(0,\infty)$ and satisfies the relations $U_s(a) = U_s'(a) = 0$ together with $U_s''(a) \neq  0$, for a given $a\in (0, \infty)$. Let $\tau\in \mathbb C$ be as in \Cref{prop:main-prop} and denote by $\sigma_0:=-{\mathcal{I}m(\tau)}>0$. 
    For any $k\in \mathbb N$ and any $\sigma \in (0, \sigma_0)$ the following inequality holds true
    \begin{equation}\label{est:lemma-Tk-semigrouo-sup}
        \sup_{0\leq t \leq \frac{1}{3(\sigma_0-\sigma)}
        \frac{\ln(k)}{\sqrt[3]{k}}
        }
        e^{-\sigma t k^\frac{1}{3}}
        \big\|  T_k(t) 
        \big\|_{\mathcal{L}\big(
        W^{1, \infty}_\alpha 
        \times 
        W^{0, \infty}_\alpha,  W^{0, \infty}_\alpha  \big) }
        >
        \mathcal{C}_{\sigma}
        k^\frac{1}{3}\quad
        \text{with}\quad 
        \mathcal{C}_{\sigma}
        :=\frac{\smallc}{2}
        \frac{ (\sigma_0-\sigma)}{(\sigma_0-\sigma) + \mathcal{C}_R},
    \end{equation}
    where the constants $\smallc$ and $\mathcal{C}_R$ are as in \Cref{prop:main-prop}.
\end{lemma}
\begin{remark}
    Before establishing the proof of \Cref{lemma:inflation-of-Tk}, some remarks are here in order. Inequality \eqref{est:lemma-Tk-semigrouo-sup} is written in terms of the semigroup $T_k$. It implies however that there exist two profiles $u_{\rm in, k}\in W^{1, \infty}_\alpha$ and  $u_{t,\rm in, k}\in W^{0, \infty}_\alpha$ with $ \|u_{\rm in, k} \|_{W^{1, \infty}_\alpha} +\|u_{t,\rm in, k} \|_{W^{0, \infty}_\alpha} \leq 1 $ (which may differ with respect to the ones of \Cref{prop:main-prop}), such that the generated homogeneous solution $u_k$ of \eqref{LHP-k-s} satisfies
    \begin{equation}\label{est:remark-uk-inflation-homogeneous}
        \sup_{0\leq t \leq \frac{1}{3(\sigma_0-\sigma)}
        \frac{\ln(k)}{\sqrt[3]{k}}
        }
        e^{-\sigma t k^\frac{1}{3}}
        \big\|  u_k(t) 
        \big\|_{W^{0, \infty}_\alpha }
        >
        \mathcal{C}_{\sigma}
        k^\frac{1}{3}.
    \end{equation}
    Unfortunately this inequality presents a major disadvantage: it is unclear at what time the inflation $\big\|  u_k(t) 
        \big\|_{W^{0, \infty}_\alpha }\geq \mathcal{C}_{\sigma}
        k^\frac{1}{3} e^{\sigma t k^{1/3}}$ holds true. This is deeply in contrast with the instability \eqref{uniform-bound-of-ukfr} of $u^{\rm fr}_k$, which is indeed satisfied globally in time.
        Certainly, \eqref{est:remark-uk-inflation-homogeneous} is not achieved at $t = 0$ because of the initial data, however the inflation may occur at a time $t$ very close to the origin (for which $e^{-\sigma t k^{1/3}}\sim 1$). Furthermore, also in case that the inflation occurs at the largest time $t = \frac{1}{3(\sigma_0-\sigma)}
        \frac{\ln(k)}{\sqrt[3]{k}}$, we may obtain  at best  that 
        $ \|  u_k(t)   \|_{W^{0, \infty}_\alpha }\geq \mathcal{C}_{\sigma}
        k^{\frac{\sigma_0}{3(\sigma_0-\sigma)}}$, which somehow implies only an inflation of Sobolev type. In other words, because of the time limitation of estimate \eqref{est:lemma-Tk-semigrouo-sup}, we deal in this work only with ill-posedness in Sobolev spaces and our approach seems to be inconclusive in Gevrey-class $m$, with $m>3$. 
\end{remark}
\begin{proof}[Proof of \Cref{lemma:inflation-of-Tk}]
    Assume by contradiction that there exists a frequency $k  \in \mathbb N$, so that
    \begin{equation}\label{the-contradiction}
        \sup_{0\leq t \leq \frac{1}{3(\sigma_0-\sigma)}
        \frac{\ln(k)}{\sqrt[3]{k}}
        }
        e^{-\sigma t k^\frac{1}{3}}
        \big\|  T_k(t) 
        \big\|_{\mathcal{L}}
        \leq 
        \frac{\smallc}{2}
        \frac{ (\sigma_0-\sigma)
        }{(\sigma_0-\sigma) + \mathcal{C}_R}
        k^\frac{1}{3},
    \end{equation}
    where we have used the abbreviation $\mathcal L:=\mathcal{L}\big( 
        W^{1, \infty}_\alpha 
        \times W^{0, \infty}_\alpha,  W^{0, \infty}_\alpha  \big)
    $.
    We consider the global-in-time solution $u_k^{\rm fr}(t,y) = e^{i \tau k^{1/3}t} \mathbb{U}_k(y)$ 
    provided by \Cref{prop:main-prop}, which by uniqueness (at a fixed frequency) also satisfies the Duhamel's relation
    \begin{equation*}
        u_k^{\rm fr}(t,y)
        = 
        T_k(t)\big(\mathbb{U}_k, i\tau k^\frac{1}{3}\mathbb{U}_k\big)(y)
        + 
        \int_0^t
        T_k(t-s)(0, -k^\frac{4}{3}e^{i\tau k^{\frac{1}{3}}s}\mathcal{R}_k(s))(y)ds,
    \end{equation*}
    thanks to \eqref{Duhamel-formula}. Hence, applying the $W^{0, \infty}_\alpha$-norm to this identity and making use of the triangular inequality, we gather that
    \begin{align*}
        \|  T_k(t)
            &\big(\mathbb{U}_k, i\tau k^\frac{1}{3}\mathbb{U}_k\big) \|_{W^{0, \infty}_\alpha}
        \geq 
        \| u_k^{\rm fr}(t) \|_{W^{0, \infty}_\alpha } - 
        \int_0^t 
        \| 
            T_k(t-s)(0, -k^\frac{4}{3}
            e^{i\tau k^{\frac{1}{3}}s}
            \mathcal{R}_k(s)) 
        \|_{W^{0, \infty}_\alpha }
        ds
        \\
        &
        \geq
        \smallc 
        e^{\sigma_0 k^{\frac{1}{3}}t}
        - 
        \int_0^t 
        \| T_k (t-s) \|_{\mathcal{L}}
        k^\frac{4}{3}
        \| \mathcal{R}_k(s)
        \|_{W^{0, \infty}_\alpha }
        e^{\sigma_0 s k^\frac{1}{3}}
        ds
        \\
        &
        \geq
        \smallc 
        e^{\sigma_0 k^{\frac{1}{3}}t}
        - 
        \int_0^t 
        \norm{ T_k (t-s) }_{\mathcal{L}}
        e^{-\sigma k^\frac{1}{3}(t-s) }
        k^\frac{4}{3}
        \| \mathcal{R}_k(s)
        \|_{W^{0, \infty}_\alpha }
        e^{-(\sigma_0-\sigma)(t- s)k^\frac{1}{3}}
        ds
        e^{\sigma_0 t k^\frac{1}{3}},
    \end{align*}
    where we have estimated $\| u_k^{\rm fr}(t) \|_{W^{0, \infty}_\alpha }$ with the inequality in \eqref{uniform-bound-of-ukfr}. Hence applying \eqref{the-contradiction} and invoking the uniform bound \eqref{uniform-bound-on-Rk} on the forcing term $\mathcal{R}_k$, we obtain
    \begin{align*}
        \|  T_k(t)
            &\big(\mathbb{U}_k, i\tau k^\frac{1}{3}\mathbb{U}_k\big) \|_{W^{0, \infty}_\alpha}
        \geq
        \smallc 
        e^{\sigma_0 k^{\frac{1}{3}}t}
        - 
        \frac{\smallc }{2}
        \frac{(\sigma_0-\sigma)}{(\sigma_0-\sigma) + \mathcal{C}_R}
        k^\frac{1}{3}
        \mathcal{C}_R
        \int_0^t 
        e^{-(\sigma_0-\sigma)(t-s) k^\frac{1}{3}}
        ds
        e^{\sigma_0 k^\frac{1}{3} t}.
    \end{align*}
    Multiplying both l.~and r.h.s.~by $e^{-\sigma_0 k^{1/3}t }/\small c$ and calculating explicitely the integral on the r.h.s, we get
    \begin{align*}
        \frac{e^{-\sigma_0 k^\frac{1}{3}t }}{\small c}
        \|  T_k(t)
            &\big(\mathbb{U}_k, i\tau k^\frac{1}{3}\mathbb{U}_k\big) \|_{W^{0, \infty}_\alpha}
        \geq
        1
        - 
        \frac{1}{2}
        \frac{\mathcal{C}_R}{(\sigma_0-\sigma) + \mathcal{C}_R}
        \left(
                1
                -
                e^{-(\sigma_0-\sigma)k^\frac{1}{3}t}
        \right).
    \end{align*}
    Recasting $e^{-\sigma_0 k^\frac{1}{3}t } = e^{-(\sigma_0-\sigma) k^\frac{1}{3}t }e^{-\sigma k^\frac{1}{3}t } $ on the l.h.s., we apply once more \eqref{the-contradiction}, to deduce that 
    \begin{align*}
        \frac{1}{2}
        \frac{ (\sigma_0-\sigma)
        }{(\sigma_0-\sigma) + \mathcal{C}_R}
        e^{-(\sigma_0-\sigma) k^\frac{1}{3}t }
        k^\frac{1}{3}
        \geq
        1
        - 
        \frac{1}{2}
        \frac{\mathcal{C}_R}{(\sigma_0-\sigma) + \mathcal{C}_R}
        \left(
                1
                -
                e^{-(\sigma_0-\sigma)k^\frac{1}{3}t}
        \right),
    \end{align*}
    for any time $t \in \big[0, \frac{1}{3(\sigma_0-\sigma)}\frac{\ln(k)}{\sqrt[3]{k}}\big]$. 
    We hence set $ t =\frac{1}{3(\sigma_0-\sigma)}
    \frac{\ln(k)}{\sqrt[3]{k}} $, which in particular implies that
    $e^{-(\sigma_0-\sigma)k^{1/3}t} = k^{-1/3}$, so that    
     \begin{align*}
        \frac{1}{2}
        \frac{ (\sigma_0-\sigma)}{(\sigma_0-\sigma) + \mathcal{C}_R}
        \geq
        1
        - 
        \frac{1}{2}
        \frac{\mathcal{C}_R(
                1
                -
               k^{-\frac{1}{3}})}{(\sigma_0-\sigma) + \mathcal{C}_R}
        .
    \end{align*}
    By bringing the last term on the r.h.s~to the l.h.s., we finally obtain that 
    \begin{equation*}
        \frac{1}{2}
        =
        \frac{1}{2}
        \frac{ (\sigma_0-\sigma)+\mathcal{C}_R}{(\sigma_0-\sigma) + \mathcal{C}_R}
        \geq 
        \frac{1}{2}
        \frac{ (\sigma_0-\sigma)+\mathcal{C}_R (
                1
                -
               k^{-\frac{1}{3}})}{(\sigma_0-\sigma) + \mathcal{C}_R}
               \geq 1,
    \end{equation*}
    which is indeed a contradiction.
    This concludes the proof of \Cref{lemma:inflation-of-Tk}.
\end{proof}

\subsection*{(ii) Proof of \Cref{thm:ill-posedness-Sobolev} and the inflation of the Sobolev norms}
Thanks to \Cref{lemma:inflation-of-Tk}, 
we are now in the condition to conclude the proof of \Cref{thm:ill-posedness-Sobolev}. Let $\sigma $ be a fixed value in $(0, \sigma_0) = (0, -{\mathcal{I}m(\tau)})$ and let $\delta>0$ be an arbitrary short time. We first recall \eqref{est:lemma-Tk-semigrouo-sup}
\begin{equation}\label{est:final-proof-1}
        \sup_{0\leq t \leq \frac{1}{3(\sigma_0-\sigma)}
        \frac{\ln(k)}{\sqrt[3]{k}}
        }
        e^{-\sigma t k^\frac{1}{3}}
        \big\|  T_k(t) 
        \big\|_{\mathcal{L} }
        >
        \mathcal{C}_\sigma k^\frac{1}{3},
\end{equation}
for any $k\in \mathbb N$, where $\mathcal{L}$ abbreviates $\mathcal{L}\big(W^{1, \infty}_\alpha \times W^{0, \infty}_\alpha,  W^{0, \infty}_\alpha  \big)$. 
Recalling that $\mu \in [0, 1/3)$, we consider a general frequency $k \in \mathbb N$ that satisfies
\begin{equation*}
    \frac{1}{3(\sigma_0-\sigma)}
    \frac{\ln(k)}{\sqrt[3]{k}} \leq \delta,
    \quad 
    \text{and}
    \quad 
    \mathcal{C}_\sigma k^{\frac{1}{3}-\mu}
    \geq \frac{1}{\delta}.
\end{equation*}
By multiplying \eqref{est:final-proof-1} with $k^{-\mu}$, we remark that there exists a time $t_\delta \in (0, \delta)$, such that
\begin{equation*}
        k^{-\mu}
        \big\|  T_{k} (t_\delta) 
        \big\|_{\mathcal{L}  }
        \geq 
        k^{-\mu}
        e^{-\sigma t_\delta k^{\frac{1}{3}}}
        \big\|  T_{k} (t_\delta) 
        \big\|_{\mathcal{L} }
        >
        \mathcal{C}_\sigma k^{\frac{1}{3}-\mu}
        \geq 
        \frac{1}{\delta}.
\end{equation*} 
Next we write the above estimate in terms of a specific solution $u_k= u_k(t,y)$ of the homogeneous system \eqref{LHP-k-s}: we consider two profiles $u_{\rm in, k}$ and $ u_{t,\rm in, k}$  in $W^{1,\infty}_\alpha$ and $W^{0, \infty}_\alpha$, which satisfy 
\begin{equation*}
    \| u_{k,{\rm in},k} \|_{W^{1,\infty}_\alpha}
    +
    \| u_{t,{\rm in},k} \|_{W^{0,\infty}_\alpha }
    \leq  
    1
    \quad 
    \text{and}
    \quad 
    k^{-\mu}
    \big\|  u_k(t_\delta)
        \big\|_{W^{0, \infty}_\alpha }
        >
        \frac{1}{\delta},
    \quad 
    \text{where}
    \quad 
    u_k(t) :=
    T_{k} (t) 
    (u_{{\rm in},k}, u_{t,{\rm in},k}).
\end{equation*}
We hence set a solution $u = u(t,x,y)$ with of the original system \eqref{LHP} by means of
\begin{equation*}
    u(t,x,y) := 
    \frac{1}{ (1+k^2)^{m/2}}e^{i k  x }u_k(t,y),
    \qquad 
    (t,x,y)\in (0,T)\times \mathbb T\times \mathbb R_+.
\end{equation*}
In particular, the initial data of $u$ satisfy
\begin{equation*}
   \| u(0,\cdot ) \|_{H^m W^{1, \infty}_\alpha}
   + 
   \| \partial_t u(0, \cdot ) \|_{H^m W^{0, \infty}_\alpha}
   =
   \| u_{k,{\rm in},k} \|_{W^{1,\infty}_\alpha}
   +
   \| u_{t,{\rm in},k} \|_{W^{0,\infty}_\alpha }
   \leq  
   1.
\end{equation*}
On the other hand, the following inflation of the Sobolev norm holds true at $t = t_\delta$:
\begin{align*}
    \big\| 
        u(t_\delta)
    \big\|_{H^{m-\mu}W^{0, \infty}_\alpha}
    = 
    k^{-\mu}
    \big\| 
        u_k(t_\delta)
    \big\|_{W^{0, \infty}_\alpha}
    >
    \frac{1}{\delta}.
\end{align*}
This concludes the proof of \Cref{thm:ill-posedness-Sobolev}.

\section{Proof of \Cref{prop:main-prop}}\label{sec:proof-of-proposition-inhomogeneous}
Without loss of generality, we may assume that $U_s''(a)<0$. Contrarily, we might set 
$\tilde U_s(y) := -U_s(y)$, $\tilde u(t,x,y) := -u(t,-x,y)$, $\tilde v(t,x,y):= v(t,-x,y)$ and $\tilde f(t,x,y) := -f(t,-x,y)$. Thus $(\tilde u,\tilde v)$ is solution of 
\begin{equation*}
     ( \partial_t +1)\big(
    \partial_t \tilde u  + \tilde  U_{{\rm s}}  \partial_x \tilde  u + \tilde  v  \tilde U'_{{\rm s}}\big) - \partial_{y}^2 \tilde u =\tilde f,\qquad 
    \text{with}\quad \tilde U_s''(a)<0.
\end{equation*}
For simplicity, we denote by $\varepsilon= \ee(k) :=1/k > 0$ the inverse of a positive frequency $k\in \mathbb N$, so that $\ee \to 0$ when $k\to \infty$. Furthermore, throughout our proof, we will repeatedly use the following abuse of notation: we interchange any index $k$ of a general function with its corresponding index $\ee$, for instance
\begin{equation*}
    u_\ee = u_k,\quad v_\ee = v_k,\quad \mathbb V_\ee = \mathbb V_k, \quad \mathcal{R}_\ee = \mathcal{R}_k,
\end{equation*}
and so on.

\noindent 
We aim to construct a solution $(u_\ee(t,x,y), v_\ee(t,x,y))$  of \eqref{LHP-k}
with a forcing term $f_\ee(t,x,y)$ of the form
\begin{equation}\label{def:forms_of_u_v}
    u_\ee(t,x,y) = e^{i \frac{\omega(\ee)t+x}{\ee }}\mathbb{U}_\ee(y),\quad 
    v_\ee(t,x,y) =\frac{1}{\ee} e^{i \frac{\omega(\ee)t+x}{\ee }}\mathbb{V}_\ee(y),\quad 
    f_\ee(t,x,y) = 
    -\frac{1}{\ee^\frac{4}{3}} 
    e^{i \frac{\omega(\ee)t+x}{\ee }}
    \mathcal{R}_\ee(y),
\end{equation}
for a suitable $\omega(\ee) \in \mathbb C$ that we will soon determine and suitable profiles $\mathbb U_\ee,\mathbb V_\ee, \mathcal{R}_\ee :[0, \infty) \to \mathbb C$.
We momentarily take for given that $\omega$ is $\mathcal{O}(\ee^\frac{2}{3}) $ and 
$\mathcal{R}_\ee \in W^{0, \infty}_\alpha(\mathbb R_+)$, for any $\ee>0$, with 
$\| \mathcal{R}_\ee \|_{ W^{0, \infty}_\alpha} $ in $\mathcal{O}(\ee^\frac{4}{3})$ 
(for the impatient reader, check \eqref{def-omega} and \eqref{eq:explicit-Rk}).

\noindent 
In reality, the profile $\mathbb U_\ee$ is redundant, since the divergence-free condition implies that
\begin{equation*}
    \partial_x u_\ee + \partial_y v_\ee = 0
    \quad \Rightarrow\quad 
    \frac{ i}{\ee }  e^{i \frac{\omega(\ee)t+x}{\ee }}\mathbb U_\varepsilon(y) + \frac{1}{\ee}  
    e^{i \frac{\omega(\ee)t+x}{\ee }}\mathbb V_\ee'(y)= 0
    \quad 
    \Rightarrow 
    \quad 
    \boxed{\mathbb U_\ee(y) = i\mathbb V_\ee'(y)}.
\end{equation*}

\noindent
We hence plug the expressions \eqref{def:forms_of_u_v} into the main equation \eqref{LHP}. We obtain that $\mathbb U_\ee$ and $\mathbb V_\ee$ satisfies
\begin{equation*}
    e^{i \frac{\omega(\ee)t+x}{\ee }}
    \bigg\{
    \bigg(
        \frac{i\omega(\ee)}{\ee} + 1
    \bigg)
    \bigg( 
        \frac{i\omega(\ee)}{\ee}\mathbb U_\ee(y) + 
        \frac{i}{\ee}
        U_{\rm s}(y)
        \mathbb U_\ee(y)
        +
        \frac{1}{\ee}
        U_s'(y) \mathbb V_\ee(y)
    \bigg)
    -\mathbb U_\ee ''(y) 
    \bigg\}
    = 
    -\frac{1}{\ee^\frac{4}{3}} 
    e^{i \frac{\omega(\ee)t+x}{\ee }}
    \mathcal{R}_\ee(y).
\end{equation*}
Dividing by $e^{i \frac{\omega(\ee)t+x}{\ee }}$ and multiplying by $-\ee^{4/3}$ we get
\begin{equation*}
    \bigg(
        \frac{i\omega(\ee)}{\ee^\frac{2}{3}} + \ee^{\frac{1}{3}}
    \bigg)
    \bigg( 
        \omega(\ee)(- i \mathbb U_\ee(y)) + 
        U_{\rm s}(y)
        (-i \mathbb U_\ee(y))
        - U_s'(y)  \mathbb V_\ee(y)
    \bigg)
    + \ee^{\frac{4}{3}}  \mathbb U_\ee ''(y) 
    = \mathcal{R}_\ee(y).
\end{equation*}
Thus recasting the above relation uniquely in terms of $\mathbb V_\ee$ (with $\mathbb V_\ee' = -i \mathbb U_\varepsilon $), we finally derive the following ordinary differential equation:
\begin{equation}\label{eq:v-eps}
    \bigg(
        \frac{i\omega(\ee)}{\ee^\frac{2}{3}} + \ee^{\frac{1}{3}}
    \bigg)
    \Big( 
        \big(
            \omega(\ee)  
         +
         U_{\rm s}(y)
        \big)
        \mathbb V_\ee'(y)- 
        U_s'(y) \mathbb V_\ee(y)
    \Big)
    +i\ee^{\frac{4}{3}}  \mathbb V_\ee '''(y) 
    = \mathcal R_\ee(y),\qquad 
    y > 0.
\end{equation}
We shall remark that for any $\mathbb V_\ee$ satisfying \eqref{eq:v-eps} in distributional sense, in reality  $\mathbb V_\ee'$ belongs to $W^{2, \infty}_\alpha (\mathbb R_+)$ and  $\mathbb V_\ee\in W^{3, \infty}(\mathbb R_+)$ ($\mathbb V_\ee$ does not need to decay as $y\to \infty$), for any $\ee>0$, since $U_s\in W^{4, \infty}_\alpha $ and $\mathcal R_\ee\in W^{0, \infty}_\alpha$.  
Formally, by sending $\ee \to  0$, denoting the asymptotic $\mathbb V_\ee \to v_a$, and recalling that the limit $\lim\limits_{\ee \to 0 } \ee^{-2/3}\omega(\ee) \neq 0$, we obtain the equation
\begin{equation}\label{def:va}
       U_{\rm s}(y)
       v_a'(y)- 
       U_s'(y) v_a(y) 
        = 0,\qquad 
        y > 0. 
\end{equation}
Of course this limit is only formal, we shall soon analyse the difference between a solution $\mathbb V_\ee $ of \eqref{eq:v-eps} and a solution $v_a$ of \eqref{def:va}.
Since $U_s(a) =U_s'(a) =  0$, for a given $a\in (0, \infty)$, \eqref{def:va} admits a non-trivial weak solution given by
\begin{equation*}
   \boxed{ v_a(y) = H(y-a) U_s(y)}.
\end{equation*}
We remark that $v_a$ belongs to $W^{2,\infty}_\alpha(\mathbb R_+)\cap W^{3,\infty}_\alpha(\mathbb R_+\setminus\{ a\} )$, 
since $U_s'(a) = U_s(a) = 0$. However, $v_a'''$ behaves as a Dirac delta distribution in $y = a$, since $U_s''(a) \neq 0$.

\noindent
We next set a complex number $\tau \in \mathbb C\setminus \{ 0 \}$ and we introduce the following Ansatz:
\begin{equation}\label{def-omega}
    \boxed{
        \omega(\ee) = \ee^\frac{2}{3}\tau
        }.
\end{equation}
We hence aim to determine $\mathbb V_\ee$ of \eqref{eq:v-eps}, making use of the following perturbation of $v_a$, depending on $\tau$:
    \begin{equation}\label{def-v-ee}
    \boxed{
    \begin{aligned}
        \mathbb{V}_\ee (y)
        &:=
        v_a(y) + \ee^\frac{2}{3} \tau H(y-a) 
        + 
        \ee^\frac{2}{3} 
        V\left(\frac{y-a}{\sqrt[3]{\ee}} \right) 
        \\
        &{\color{white}:}=
        \Big(
            U_s(y) +\ee^{\frac{2}{3}}\tau 
        \Big)
        H(y-a)
        + 
        \ee^\frac{2}{3} 
        V\left(\frac{y-a}{\sqrt[3]{\ee}} \right) 
        .
    \end{aligned}
     }
\end{equation}
Hence, the unknowns are momentarily the profile $V: \mathbb R\to \mathbb C$ and the complex number $\tau \in \mathbb C$. We remark that the expansion \eqref{def-v-ee} is similar to its homologous (2.5) in \cite{GD2010} for the Prandtl equations. There are however major differences in the exponents of the terms in $\ee$. These eventually lead to different dispersion rates: of order $\sqrt[3]{k} = \ee^{-1/3}$ for System \eqref{LHP} and of order  $\sqrt{k} = \ee^{-1/2}$ for the Prandtl equation.


\begin{remark}\label{rmk:V-has-a-jump}
Since $\mathbb{V}_\ee$ belongs to $W^{3, \infty}$, the profile $V$ must cancel the singularities  in $y = a$ of the functions $v_a(y)$ and $\ee^\frac{2}{3} \tau H(y-a)$. We will show that $V$ is indeed in $W^{3, \infty}(\mathbb R_+\setminus \{ 0 \})\cap 
L^{\infty}(\mathbb{R}_+)$ and $V'$ behave as a Dirac-delta distribution $-\tau \delta_a$ in $y = a$.
\end{remark}

\noindent
Replacing \eqref{def-omega} and \eqref{def-v-ee} into \eqref{eq:v-eps}, we gather the following equation for the profile $V$:
\begin{equation}\label{eq:first-eq-of-V}
\begin{aligned}
    \Big(
        i \tau  + \ee^{\frac{1}{3}}
    \Big)
    \bigg\{ 
        \Big(
            \varepsilon^\frac{2}{3}\tau
            &+
            U_{\rm s}(y)
        \Big)
        \bigg(
            v_a'(y)
            +
            \ee^\frac{2}{3} \tau
            \delta_a(y) + 
            \ee^{\frac{2}{3}}
            \frac{d}{d y}
            \Big[
                V\Big(\frac{y-a}{\sqrt[3]{\ee}} \Big)
            \Big]
        \bigg)
        +\\
        &- 
        U_s'(y) 
        \bigg(
            v_a(y) + \ee^\frac{2}{3} \tau H(y-a) + \ee^\frac{2}{3} 
            V\Big(\frac{y-a}{\sqrt[3]{\ee}} \Big)
        \bigg)
    \bigg\}
    +i\ee^{\frac{4}{3}}  \mathbb{V}_\ee '''(y) 
    = \mathcal{R}_\ee(y).
\end{aligned}
\end{equation}
%
We first remark that several terms of \eqref{eq:first-eq-of-V} cancel out thanks to the definition of $v_a$ in \eqref{def:va} and the conditions on the shear flow  $U_s(a) = U_s'(a) = 0$. Indeed, \eqref{eq:first-eq-of-V} can be recasted as
\begin{equation*}
\begin{aligned}
    &\Big(
        i \tau  + \ee^{\frac{1}{3}}
    \Big)
    \bigg\{ 
            \varepsilon^\frac{2}{3}\tau
            \Big(
               \underbrace{
                v_a'(y) - H(y-a)U_s'(y)
                }_{= \delta_a(y) U_s(y) = 0}
            \Big)
            +
            \ee^{\frac{4}{3}}\tau^2 \delta_a(y)
             +
            \tau   
            \ee^{\frac{4}{3}}
            \frac{ d}{ d y}
            \Big[
                V\Big(\frac{y-a}{\sqrt[3]{\ee}} \Big)
            \Big]
            +\\ 
            &+
            \Big(
            \underbrace{
                 U_{\rm s}(y)
                 v_a'(y)- 
                 U_s'(y) v_a(y) 
            }_{= 0}
            \Big)
            +
            \ee^{\frac{2}{3}}\tau 
            \underbrace{ 
               U_s(y)\delta_a(y)
            }_{y = 0}
            +
            \ee^{\frac{2}{3}}
            U_s(y)
            \frac{ d}{ d y}
            \Big[
                V\Big(\frac{y-a}{\sqrt[3]{\ee}} \Big)
            \Big]
            \!-\!
            \ee^{\frac{2}{3}}
            U_s'(y)
             V\Big(\frac{y-a}{\sqrt[3]{\ee}} \Big)
    \bigg\}+\\
    &+ i\ee^{\frac{4}{3}}  v_a '''(y) 
    + i\tau \ee^{2}  \delta_a ''(y) 
    +
    i
    \ee^{2} 
    \frac{ d^3}{ d y^3}
            \Big[
                V\Big(\frac{y-a}{\sqrt[3]{\ee}} \Big)
            \Big]
    = \mathcal{R}_\ee(y).
\end{aligned}
\end{equation*}
Next, we divide the above relation with $\ee^{4/3}>0$ and we remark that $ v_a '''(y) =  U_s''(a) \delta_a(y) + H(y-a) U_s'''(y)$. Hence, we are left with the following identity, which shall be intended in terms of distributions $\mathcal{D}'(0, \infty)$:
\begin{equation}\label{eq:last-eq-in-y-for-V-eps}
\begin{aligned}
    \Big(
        i \tau  + \ee^{\frac{1}{3}}
    \Big)
    \bigg\{ 
            \tau^2 \delta_a(y)
             +
            \tau   
            \frac{d}{d y}
            \Big[
                V\Big(\frac{y-a}{\sqrt[3]{\ee}} \Big)
            \Big]
            +
            \ee^{-\frac{2}{3}}
            U_s(y)
            \frac{d}{d y}
            \Big[
                V\Big(\frac{y-a}{\sqrt[3]{\ee}} \Big)
            \Big]
            \!-\!
            \ee^{-\frac{2}{3}}
            U_s'(y)
             V\Big(\frac{y-a}{\sqrt[3]{\ee}} \Big)
    \bigg\}+\\
    + i
    \Big( 
        U_s''(a) \delta_a(y) + H(y-a) U_s'''(y)
    \Big)
    + i\tau \ee^{\frac{2}{3}}  \delta_a ''(y) 
    +
    i
    \ee^{\frac{2}{3}}
    \frac{d^3}{d y^3}
            \Big[
                V\Big(\frac{y-a}{\sqrt[3]{\ee}} \Big)
            \Big]
    = 
    \ee^{-\frac{4}{3}}
    \mathcal{R}_\ee(y).
\end{aligned}
\end{equation}
We next apply \eqref{eq:last-eq-in-y-for-V-eps} to an appropriate test function $\varphi \in \mathcal{D}(0, \infty)$. We first consider a general test function $\psi \in \mathcal{D}(\mathbb R)$ having $\supp \psi \subseteq (-a/\sqrt[3]{\ee}, \infty)$ is valid. Hence we set $\varphi(y) = \psi((y-a)/\sqrt[3]{\ee})$, for any $y>0$. 
By applying \eqref{eq:last-eq-in-y-for-V-eps} to this specific test function $\varphi$, we get the identity
\begin{equation*}
\begin{aligned}
    \Big(
        i \tau  + \ee^{\frac{1}{3}}
    \Big)
    \Bigg\{ 
        \tau^2 \psi (0) 
        -
        \tau 
        \int_0^\infty 
         V\Big(\frac{y-a}{\sqrt[3]{\ee}} \Big)
         \psi'\Big(\frac{y-a}{\sqrt[3]{\ee}} \Big)
         \frac{dy}{\sqrt[3]{\ee}}
         -
         \ee^{-\frac{2}{3}}
          \int_0^\infty 
         V\Big(\frac{y-a}{\sqrt[3]{\ee}} \Big)
         \frac{d}{dy}
         \bigg(
            U_s(y)
            \psi\Big(\frac{y-a}{\sqrt[3]{\ee}} \Big) 
         \bigg)
         dy+\\
         -
         \ee^{-\frac{1}{3}}
         \int_0^\infty 
          U_s'(y)
          V\Big(\frac{y-a}{\sqrt[3]{\ee}} \Big)
          \psi\Big(\frac{y-a}{\sqrt[3]{\ee}} \Big) 
          \frac{dy}{\sqrt[3]{\ee}}
     \Bigg\}
     +
     i U_s''(a) \psi(0)
     +
     i
     \int_a^\infty 
     U_s'''(y)
     \psi\Big(\frac{y-a}{\sqrt[3]{\ee}} \Big) 
      dy 
      + 
      i \tau \psi''(0) +
      \\
      -
      i
      \int_0^\infty 
          V\Big(\frac{y-a}{\sqrt[3]{\ee}} \Big)
          \psi'''\Big(\frac{y-a}{\sqrt[3]{\ee}} \Big) 
          \frac{dy}{\sqrt[3]{\ee}} = 
          \ee^{-\frac{4}{3}}
          \int_0^\infty \mathcal{R}_\ee(y)
          \psi\Big(\frac{y-a}{\sqrt[3]{\ee}} \Big) 
          dy.
\end{aligned}    
\end{equation*}
We perform  a change of variables with  $\tilde{z} := (y-a)/\sqrt[3]{\ee}\in (-a/\sqrt[3]{\ee}, \infty)$. Hence $y = a+ \sqrt[3]{\ee} \tilde{z}$ and $dy =\sqrt[3]{\ee} d\tilde{z} $, so that
\begin{equation*}
\begin{aligned}
    \Big(
        i \tau  + \ee^{\frac{1}{3}}
    \Big)
    \Bigg\{ 
        \tau^2 \psi (0) 
        -
        \tau 
        \int_{-\frac{a}{\sqrt[3]{\ee}}}^\infty 
         V(\tilde{z})
         \psi' (\tilde{z})
         d\tilde{z}
         -
         \ee^{-\frac{2}{3}}
         \int_{-\frac{a}{\sqrt[3]{\ee}}}^\infty 
         V(\tilde{z})
         \frac{d}{d\tilde{z}}
         \bigg(
            U_s( a+ \sqrt[3]{\ee} \tilde{z})
            \psi(\tilde{z}) 
         \bigg)
         d\tilde{z}
         +
         \\-
         \ee^{-\frac{1}{3}}
         \int_{-\frac{a}{\sqrt[3]{\ee}}}^\infty 
         U_s'( a+ \sqrt[3]{\ee} \tilde{z})
          V(\tilde{z})
          \psi(\tilde{z})
          d\tilde{z}
     \Bigg\}
     +
     i U_s''(a) \psi(0)
     +
     i
     \sqrt[3]{\ee}
     \int_0^\infty
     U_s'''(a+\sqrt[3]{\ee} z)
     \psi(\tilde{z})d\tilde{z} + \\
      + i \tau \psi''(0) 
      -
      i\int_{-\frac{a}{\sqrt[3]{\ee}}}^\infty 
          V(\tilde{z})
          \psi'''(\tilde{z})
          d\tilde{z} = 
          \frac{1}{\ee}
          \int_{-\frac{a}{\sqrt[3]{\ee}}}^\infty  
          \mathcal{R}_\ee(a+\sqrt[3]{\ee} \tilde z)
          \psi(\tilde{z})d\tilde{z}.
\end{aligned}    
\end{equation*}
In particular, we deduce that $V$ is distributional solution in $\mathcal{D}'(-a/\sqrt[3]{\ee}, \infty)$ of the following equation:
\begin{equation}\label{eq:final-equation-for-Veps-with-rest}
\begin{aligned}
    \Big(
        i \tau  + \ee^{\frac{1}{3}}
    \Big)
    \bigg\{ 
            \tau^2 \delta_0(\tilde{z})
             +
            \tau   V'(\tilde{z})
            +
            \ee^{-\frac{2}{3}}
            U_s(a+ \sqrt[3]{\ee} \tilde{z})
            V'(\tilde{z})
             - 
            \ee^{-\frac{1}{3}} U_s'(a+ \sqrt[3]{\ee} \tilde{z})
             V(\tilde{z})
    \bigg\}
    +\\
    + 
    i U_s''(a) \delta_0(\tilde{z}) +
    i\ee^{\frac{1}{3}}   
    H(\tilde{z}) U_s'''(a+\sqrt[3]{\ee}\tilde z)  
    + i\tau  \delta_0 ''(\tilde{z}) 
    + i V^{'''}(\tilde{z})
    =  
    \ee^{-1}
    \mathcal{R}_\ee(a+\sqrt[3]{\ee} \tilde z).
\end{aligned}
\end{equation}
Remark that although $V$ does not depend on $\ee>0$, the remainder $\mathcal{R}_\ee$ is still unknown and will be chosen to abosrb all the $\ee$-dependences. We now send $\ee$ towards $0$ in order to determine an equation for the profile $V$ in $\mathcal{D}'(\mathbb R)$. We recall that we assume $\mathcal{R}_\ee\in \mathcal{O}(\ee^\frac{4}{3})$ in $W^{0, \infty}_\alpha$, that both $U_s(a)=U_s'(a) = 0$ and that $U_s\in W^{4, \infty}(0, \infty)$, hence
\begin{equation*}
\begin{alignedat}{8}
    &
    \lim_{\ee \to  0} 
    \ee^{-\frac{1}{3}}
    U_s'(a +\sqrt[3]{\ee} \tilde{z} )
    &&=
    \lim_{\ee \to  0} 
    \frac{
                U_s'(a +\sqrt[3]{\ee} \tilde{z} )-U_s'(a)
    }{\sqrt[3]{\ee}}
    = \tilde{z} U_s''(a), \\
    &\lim_{\ee \to  0} 
    \ee^{-\frac{2}{3}}
            U_s( a+ \sqrt[3]{\ee} \tilde{z})
    &&=
     \lim_{\ee \to  0} 
    \frac{
                U_s(a +\sqrt[3]{\ee} \tilde{z} )-U_s(a) - U_s'(a)\sqrt[3]{\ee} \tilde{z}
    }{\ee^{\frac23}}
    = \frac{\tilde{z}^2}{2} U_s''(a),
\end{alignedat}
\end{equation*}
for any $z> -a/\sqrt[3]{\ee}$. Since the shear flow $U_s$ is in $ W^{4, \infty}_\alpha(0, \infty)$, the convergence is uniform in any compact set of $(-a/\sqrt[3]{\ee}, \infty)$. 
Hence, as $\ee >0$ vanishes, we are left with the following identity for the profile $V $:
\begin{equation}\label{eq:V-distribution}
    \tau  
    \Big\{ 
            \tau^2 \delta_0 (\tilde{z}) +
            \Big(            
                \tau   
                +
                \frac{\tilde{z}^2}{2} U_s''(a) 
            \Big)
            V' (\tilde{z})
            - 
            \tilde{z} U_s''(a)
             V (\tilde{z})
    \Big\} + U_s''(a) \delta_0(\tilde{z}) + \tau \delta''_0(\tilde{z})+  V^{'''} (\tilde{z})= 0\quad 
    \text{in }\mathcal{D}'(\mathbb R).
\end{equation}
In particular, we aim to determine a profile $V\in W^{3, \infty} (\mathbb R_+\setminus 
\{0\})$ (decaying to $0$ as $\tilde z \to \pm \infty$), which is solution of the following linear ordinary differential equation away from the origin
\begin{equation}\label{eq:final-equation-for-V}
    \boxed{
    \tau  
    \Big\{ 
            \Big(            
                \tau   
                +
                \frac{\tilde{z}^2}{2} U_s''(a) 
            \Big)
            V' (\tilde{z})
            - 
            \tilde{z} U_s''(a)
             V (\tilde{z})
    \Big\} +  V^{'''} (\tilde{z})= 0,\qquad 
    \tilde{z}\in \mathbb R\setminus\{ 0 \},}
\end{equation}
fulfilling also the following jump relations at the origin:
\begin{equation}\label{jump-conditions}
\boxed{
\begin{aligned}
    [V]_{|\tilde{z} = 0}
    &=\lim_{h \to 0+}
    \big( 
        V(\tilde{z}+h)-V(\tilde{z}-h)
    \big)
    =-\tau,\\ 
    [V']_{|\tilde{z} = 0}
    &=\lim_{h \to 0+}
    \big( 
        V'(\tilde{z}+h)-V'(\tilde{z}-h)
    \big)=0,\\
    [V'']_{|\tilde{z}= 0}
    &=\lim_{h \to 0+}
    \big( 
        V''(\tilde{z}+h)-V''(\tilde{z}-h)
    \big)=-U_s''(a).
\end{aligned}
}
\end{equation}
We next remark that the function $\tilde{z}\in \mathbb R\setminus \{ 0 \}\to  \tau + U_s''(a)\tilde{z}^2/2$ is non-decaying solution of \eqref{eq:final-equation-for-V}. Furthermore, $\tilde{z} \in \mathbb R\to H(\tilde{z}) (\tau + U_s''(a)\tilde{z}^2/2) $ is a distributional solution of \eqref{eq:V-distribution} and it satisfies the jump conditions of \eqref{jump-conditions}. We can thus get rid of the singularity at the origin by superposition introducing the function
\begin{equation}\label{def:tildeV}
    \tilde V(\tilde{z}) = V(\tilde{z}) + H(\tilde{z})\Big( \tau + \frac{U_s''(a)}{2}\tilde{z}^2\Big),\qquad \tilde{z} \in \mathbb R.
\end{equation}
The new profile $\tilde V$ shall be determined in  $W^{3, \infty}_\mathrm{loc}(\mathbb R)$, satisfying the ODE
\begin{equation*}
\left\{
\begin{alignedat}{8}
    &\tau  
    \Big\{ 
            \Big(            
                \tau   
                +
                \frac{\tilde{z}^2}{2} U_s''(a) 
            \Big)
            \tilde V' (\tilde{z})
            - 
            \tilde{z} U_s''(a)
            \tilde V (\tilde{z})
    \Big\} + \tilde V^{'''} (\tilde{z})
    =  0,
    \qquad \tilde{z} \in \mathbb R,\\
    &\lim_{\tilde{z}\to -\infty} \tilde V (\tilde{z}) = 0,\qquad 
     \lim_{\tilde z \to +\infty}
     \bigg(
        \tilde V (\tilde{z})
        -
        \Big(
            \tau + \frac{U_s''(a)}{2}\tilde{z}^2
        \Big)
    \bigg) = 0.
\end{alignedat}
\right.
\end{equation*}
Next, recalling that we seek for a $\tau \in \mathbb C$ with $\mathcal{I}m(\tau) <0$, we introduce the function 
\begin{equation}\label{def:tildeW}
   \tilde{W}(\tilde{z}) := \frac{\tilde V(\tilde{z})}{\tau + U_s''(a)\frac{\tilde{z}^2}{2}},\qquad \tilde{z}\in \mathbb R.
\end{equation}
We thus get that $W \in W^{3, \infty}(\mathbb R)$ shall satisfy
\begin{equation}\label{eq:equation-of-W}
    \boxed{
    \tau 
    \bigg( 
        \tau + U_s''(a)\frac{\tilde{z}^2}{2}
    \bigg)^2
    \tilde{W}'(\tilde{z})
    +
    \frac{d^3}{d\tilde{z}^3}
    \bigg[ 
        \Big(
          \tau + U_s''(a)\frac{\tilde{z}^2}{2}
        \Big)\tilde{W}(\tilde{z}) 
    \bigg]= 0,\qquad \tilde{z}\in \mathbb R,
    }
\end{equation}
with boundary conditions
\begin{equation*}
    \lim_{\tilde{z} \to -\infty} \tilde{W}(\tilde{z}) = 0,\qquad 
    \lim_{\tilde{z} \to +\infty} \tilde{W}(\tilde{z}) = 1.
\end{equation*}
Finally, we perform the following change of variables 
\begin{equation}\label{def:W-gamma-z}
    \tau  = \bigg(\frac{|U_s''(a)|}{2} \bigg)^\frac{1}{3}\gamma ,\qquad 
    \tilde{z} = \bigg(\frac{2}{|U_s''(a)|} \bigg)^{\frac{1}{3}} z,\qquad 
    W(z) = 
    \tilde{W} \bigg(
        \sqrt[3]{\frac{2}{|U_s''(a)|}} z
    \bigg),
\end{equation}
which leads to 
\begin{equation*}
    \gamma 
    \bigg( 
        \gamma + {\rm sgn}\,U_s''(a) z^2
    \bigg)^2
    W'(z)
    +
    \frac{d^3}{d  z^3}
    \bigg[ 
        \Big(
           \gamma  + {\rm sgn}\,U_s''(a) z^2
        \Big) W(z) 
    \bigg]= 0,\qquad z\in \mathbb R.
\end{equation*}
Recalling that we assume $U_s''(a)<0$, we finally obtain  
\begin{equation}\label{eq:equation-of-W-tilde}
    \boxed{
    \gamma
    \big( 
        \gamma - z^2
    \big)^2
    \frac{d}{d z} W(z)
    +
    \frac{d^3}{d z^3}
    \bigg[ 
        \big(
           \gamma - z^2
        \big)W(z) 
    \bigg]= 0,\qquad z\in \mathbb R
    }
\end{equation}
with boundary conditions
\begin{equation}\label{bdy-cdts-of-W}
    \lim_{ z \to -\infty}  W( z) = 0,\qquad 
    \lim_{ z \to +\infty}  W( z) = 1.
\end{equation}

\subsection{The spectral condition}
The goal of this section is to show that the ordinary differential equation \eqref{eq:equation-of-W-tilde}, together with its boundary conditions \eqref{bdy-cdts-of-W}, admits a smooth solution $W=W(z)$ for a suitable  $\gamma\in \mathbb C$, whose imaginary part is negative. We proceed with a similar procedure as the one used by G\'erard-Varet and Dormy in \cite{GD2010}. We shall however remark that \eqref{eq:equation-of-W-tilde} inherently differs from its counterpart of \cite{GD2010} (cf.~(1.7) in \cite{GD2010}), since the authors derived an ODE of the form 
\begin{equation}\label{eq:GV-Dormy-equation-of-W}
    \big( 
        \gamma - z^2
    \big)^2
    \frac{d}{d z}
    W(z)
    +
    i 
    \frac{d^3}{d z^3}
    \bigg[ 
        \big(
           \gamma - z^2
        \big)W(z) 
    \bigg]= 0,\qquad z\in \mathbb R.
\end{equation}
Indeed, note  that by dividing \eqref{eq:equation-of-W-tilde} by $\gamma$, the leading third derivative in \eqref{eq:equation-of-W-tilde} is multiplied by $1/\gamma$, while in \eqref{eq:GV-Dormy-equation-of-W} the third derivative relates uniquely to $i$. Nevertheless, our aim is to find a $\gamma\in \mathbb C$ with $\mathcal{I}m(\gamma)<0$, thus \eqref{eq:equation-of-W-tilde} and \eqref{eq:GV-Dormy-equation-of-W} share eventually similarities in terms of the behaviour of their solutions. This is exploited in details in what follows.

\noindent 
To begin with, we consider the auxiliary eigenvalue problem (cf.~(3.2) in \cite{GD2010}):
\begin{equation}\label{eq:eigenvalue-problem}
        Af(\tilde{x}) := \frac{1}{\tilde{x}^2+1}f''(\tilde{x}) + \frac{6 \tilde{x}}{(\tilde{x}^2+1)^2}f'(\tilde{x}) + \frac{6}{(\tilde{x}^2+1)^2}
       f(\tilde{x}) = \alpha\, f(\tilde{x}),\qquad \tilde{x}\in \mathbb R.
\end{equation}
We aim to build the solution $W$ in terms of the eigenfunction $f$. In order to do so,  let us first recall certain of its underlying properties. The domain $D(A)$ of the operator $A$ is defined by making use of the weighted spaces
\begin{alignat*}{8}
    \mathcal{L}^2
    &:=
    \bigg\{
    f\in L^2_{\rm loc}(\mathbb R),\; \int_\mathbb{R} (\tilde{x}^2+1)^4 |f(\tilde{x})|^2d(\tilde{x}) < + \infty 
    \bigg\},
    \\
    \mathcal{H}^1
    &:=
    \bigg\{
    f\in H^1_{\rm loc}(\mathbb R),\;
     \int_\mathbb{R} (\tilde{x}^2+1)^4 |f(\tilde{x})|^2d(\tilde{x})+
     \int_\mathbb{R} (\tilde{x}^2+1)^3 |f'(\tilde{x})|^2d(\tilde{x})< + \infty 
    \bigg\},
    \\
    D(A)
    &:=
    \Big\{ f \in  \mathcal{H}^1,\; Af \in \mathcal{L}^2 \Big\}.
\end{alignat*}
Furthermore, the function $f$ admits an extension on a simply connected domain of $\mathbb C$ and decays exponentially along suitable sectors. 
\begin{lemma}\label{lemma:eigenvalue}
    The problem \eqref{eq:eigenvalue-problem} admits a positive eigenvalue $\alpha>0$ and an eigenfunction $f\in D(A)$. Secondly, the function $f$ admits a complex extension, which still satisfies \eqref{eq:eigenvalue-problem} and is holomorphic in the simply connected domain
    \begin{equation*}
        \mathbb{C} \setminus \set{i \beta \in \mathbb C \ \middle| \ \beta > \sqrt{\alpha}\text{ or }\beta < - \sqrt{\alpha}  }.
    \end{equation*}
    Furthermore, in the sectors ${\rm arg}(\tilde{x})\in (-\pi/4 + \delta, \pi/4-\delta)$ and ${\rm arg}(  \tilde{x})\in (3\pi/4 + \delta, 5\pi/4-\delta)$ for $\delta >0$, it satisfies the inequality 
    \begin{equation}\label{ineq:f-exp}
        |f( \tilde{x})|\leq C_\delta \exp\pare{ -  \frac{\sqrt{\alpha} \av{ \tilde{x}}^2}{4 } }.
    \end{equation}
\end{lemma}
\begin{proof}
This lemma follows from Proposition 2 in \cite{GD2010}. Indeed defining $\theta:= -\alpha^{1/2}$ and $\tilde{z} = \alpha^{1/4}\tilde{x}$ , $Y(\tilde{z}) = f(\tilde{x})$, the function $Y$ is solution of 
\begin{equation*}
    (\theta - \tilde{z}^2)Y''(\tilde{z}) - 6 \tilde{z} Y'(\tilde{z}) + ( (\theta-\tilde{z}^2)^2-6 )Y(\tilde{z}) = 0, 
\end{equation*}
which corresponds to equation (3.5) in \cite{GD2010}. \\
Thanks to Proposition 2 in \cite{GD2010}, $Y$ is holomorphic in the domain $U_\theta:=\mathbb{C} \setminus \ \set{i \beta \in \mathbb C \ \middle| ~\beta > \sqrt{|\theta|}\text{ or }\beta < - \sqrt{|\theta|}  }$ and satisfies
\begin{equation*}
    |Y(\tilde{z})|\leq C_\delta \exp\bigg(-  \frac{|\tilde{z}|^2}{4} \bigg)\,
\end{equation*}
  in the sectors ${\rm arg}(\tilde{z})\in (-\pi/4 + \delta, \pi/4-\delta)$ and ${\rm arg}(\tilde{z})\in (3\pi/4 + \delta, 5\pi/4-\delta)$. The result follows by a direct substitution.
\end{proof}

\noindent 
We are now in the condition of defining the complex number $\gamma\in \mathbb C$ of \eqref{eq:equation-of-W-tilde} in terms of the positive eigenvalue $\alpha$ of \Cref{lemma:eigenvalue}. Recalling that the eigenvalue $\alpha$ is positive, we set
\begin{equation}\label{def:gamma}
    \boxed{\gamma := \alpha^{\frac{1}{3}}e^{-i\frac{2\pi}{3}}} \quad \Rightarrow 
    \quad \mathcal{I}m(\gamma) = -\frac{\sqrt{3}}{\,2}\alpha^{\frac{1}{3}}<0.
\end{equation}
This also implies  that  the complex number $\tau = (|U_s''(a)|/2)^{1/3}\gamma$ in \eqref{def:W-gamma-z} has negative imaginary part, i.e.\
\begin{equation*}
    \boxed{\tau :=\bigg(\frac{|U_s''(a)|}{2} \alpha\bigg)^{\frac{1}{3}}e^{-i\frac{2\pi}{3}}} \quad \Rightarrow 
    \quad \mathcal{I}m(\tau) = -\frac{\sqrt{3}}{\,2}\bigg(\frac{|U_s''(a)|}{2} \alpha\bigg)^{\frac{1}{3}}<0.
\end{equation*}
We next introduce the change of variable w.r.t.\ \eqref{eq:eigenvalue-problem}, 
\begin{equation*}
    z =\alpha^{\frac{1}{6}}e^{i\frac{\pi}{6}} \tilde{x} \in 
    \Omega := \mathbb{C} \setminus \set{ e^{i 2\pi/3} \beta \in \mathbb C,\text{ with }\beta >  \alpha^{2/3}\text{ or }\beta < - \alpha^{2/3} },
\end{equation*}
and the function $X:\Omega \to \mathbb C$ by means of 
\begin{equation}\label{def:X}
    X(z) = f(\tilde{x}) = f( \alpha^{-\frac{1}{6}}e^{-i\frac{\pi}{6}}z).
\end{equation}
\begin{remark} 
We shall remark that $X$ decays exponentially as $z\in \mathbb R$ converges towards $\pm \infty$. Indeed, in virtue of \eqref{ineq:f-exp}, we gather
that
\begin{equation}\label{eq:decay-of-X}
    |X(z)| \leq C \exp\bigg( - \sqrt{\alpha} \ \frac{|\alpha^{-\frac{1}{6}}e^{-i\frac{\pi}{6}}z|^2}{4} \bigg) 
    = 
    C \exp\pare{ - \frac{\alpha^\frac{1}{6}|z|^2}{4} } 
    =
    C \exp\pare{ - \frac{\sqrt{\av{\gamma}} |z|^2}{4} }. 
\end{equation}
\end{remark}
\noindent 
The corresponding derivatives of $X$ are trivially computed:
\begin{alignat*}{8}
    X'(z)  &= \alpha^{-\frac{1}{6}}e^{-i\frac{\pi}{6}} f'( \alpha^{-\frac{1}{6}}e^{-i\frac{\pi}{6}}z),\\
    X''(z) &= \alpha^{-\frac{1}{3}}e^{-i\frac{\pi}{3}} f''( \alpha^{-\frac{1}{6}}e^{-i\frac{\pi}{6}}z).
\end{alignat*}
We can now recast the relation \eqref{eq:eigenvalue-problem} of the eigenfunction $f$ in terms of $X$ and its derivatives. First
\begin{align*}
    &\alpha X(z) 
    =
    \alpha f(\alpha^{-\frac{1}{6}}e^{-i\frac{\pi}{6}}z) 
    =
    Af (\alpha^{-\frac{1}{6}}e^{-i\frac{\pi}{6}}z) 
    \\
    &=
     \frac{1}{\alpha^{-\frac{1}{3}}e^{-i\frac{\pi}{3}}z^2+1}
     f''(\alpha^{-\frac{1}{6}}e^{-i\frac{\pi}{6}}z)  
     + 
     \frac{6 \alpha^{-\frac{1}{6}}e^{-i\frac{\pi}{6}}z}{(\alpha^{-\frac{1}{3}}e^{-i\frac{\pi}{3}}z^2+1)^2}
     f'(\alpha^{-\frac{1}{6}}e^{-i\frac{\pi}{6}}z) 
     + 
     \frac{6}{(\alpha^{-\frac{1}{3}}e^{-i\frac{\pi}{3}}z^2+1)^2}
     f(\alpha^{-\frac{1}{6}}e^{-i\frac{\pi}{6}}z)
     \\
    &=
     \frac{\alpha^{\frac{1}{3}}e^{i\frac{\pi}{3}}}{\alpha^{-\frac{1}{3}}e^{-i\frac{\pi}{3}}z^2+1}
     X''(z)
     + 
     \frac{6z}{(\alpha^{-\frac{1}{3}}e^{-i\frac{\pi}{3}}z^2+1)^2}
     X'(z)
     + 
     \frac{6}{(\alpha^{-\frac{1}{3}}e^{-i\frac{\pi}{3}}z^2+1)^2}
     X(z)
     \\
    &=
     \frac{\alpha^{\frac{2}{3}}e^{i\frac{2\pi}{3}}}{z^2-\alpha^{\frac{1}{3}}e^{-i\frac{2\pi}{3}}}
     X''(z)
     + 
     \frac{6\alpha^{\frac{2}{3}}e^{i\frac{2\pi}{3}}z}{(z^2- \alpha^{\frac{1}{3}}e^{-i\frac{2\pi}{3}})^2}
     X'(z)
     + 
     \frac{6\alpha^{\frac{2}{3}}e^{i\frac{2\pi}{3}}}{(z^2- \alpha^{\frac{1}{3}}e^{-i\frac{2\pi}{3}})^2}
     X(z).
\end{align*}
We next multiply both left and right-hand side with $\alpha^{-\frac{2}{3}}e^{-i\frac{2\pi}{3}}(z^2- \alpha^{\frac{1}{3}}e^{-i\frac{2\pi}{3}})^2$, so that
\begin{equation*}
    \alpha^{\frac{1}{3}}e^{-i\frac{2\pi}{3}}
    (z^2- \alpha^{\frac{1}{3}}e^{-i\frac{2\pi}{3}})^2
    X(z)
    = 
    (z^2- \alpha^{\frac{1}{3}}e^{-i\frac{2\pi}{3}})
    X''(z)+
    6zX'(z)
    +
    6X(z).
\end{equation*}
Recalling that $\gamma = \alpha^{\frac{1}{3}}e^{-i\frac{2\pi}{3}}$ from \eqref{def:gamma}, we finally obtain the following relation on $X$,
\begin{equation*}
    (z^2- \gamma)^2
    X(z)
    = 
    (z^2- \gamma)
    X''(z)+
    6zX'(z)
    +
    6X(z),\qquad z \in \Omega,
\end{equation*}
which can also be written as
\begin{equation*}
    \gamma
    (\gamma-z^2)^2
    X(z)
    = 
    -\frac{d^3}{dz^3}
    \Big[
        (\gamma-z^2)
        \int_{-\infty}^zX(s)ds
    \Big].
\end{equation*}
In this last identity, we have used that $X$ is in $\mathbb R$ integrable, thanks to \eqref{def:X} and the exponential decay~\eqref{ineq:f-exp}. We hence aim to define the function $W:\mathbb R \to \mathbb C$ as
\begin{equation}\label{eq:relation-between-W-and-X}
    \boxed{ W(z) := \frac{\int_{-\infty}^z X(s) ds}{\int_{\mathbb R} X(s) ds} }.
\end{equation}
 For this definition, we need therefore that $X$ has average in $\mathbb R$ different from $0$. Assume by contradiction that this were false. Then the function $F:\mathbb R \to \mathbb C$ defined by 
\begin{equation*}
    F(z) := (\gamma - z^2)\int_{-\infty}^z X(s) ds 
\end{equation*}
decays exponentially as $z\to \pm \infty$. Furthermore, $F$ satisfies
\begin{equation*}
    F^{(4)}+ \gamma (\gamma-z^2)F''(z)+ 2\gamma F(z) = 0.
\end{equation*}
Multiplying the equation by $\overline{F''(z)}$ and integrating over $\mathbb R$, we gather
\begin{equation*}
  \gamma  \int_{\mathbb R}(\gamma-z^2)|F ''(z)|^2dz - \int_{\mathbb R}|F '''(z)|^2dz -
    2\gamma \int_{\mathbb R}|F '(z)|^2dz = 0.
\end{equation*}
Hence extrapolating the imaginary part
\begin{equation*}
   - \frac{\mathcal{I}m \pare{\gamma^2}}{ \mathcal{I}m\pare{\gamma}} \int_{\mathbb R}|F''(z)|^2dz  
   + \int_{\mathbb R}\pare{z^2|F''(z)|^2  +2 |F'(z)|^2}dz  = 0.
\end{equation*}
Since $\mathcal{I}m(\gamma) <0$, cf. \cref{def:gamma}, one has that $F' \equiv 0$ which implies that also $F\equiv 0$, from the exponential decay. We conclude that $X$ is identically null and so also the holomorphic function $f$ in \eqref{eq:eigenvalue-problem}. Since $f$ is an eigenfunction, this brings a contradiction. 

\subsection{The velocity profile and the remainder}\label{sec:the-reminder}
In virtue of the previous sections, we are now in the condition to determine both $\mathbb U_k$ and the reminder $\mathcal{R}_k$, introduced in \Cref{prop:main-prop}. 
We first recast the profile $V$ in \eqref{def-v-ee} in terms of the complex number $\tau = \gamma (|U''_s(a)|/2)^{1/3}$ and the meaningful solution $W$ of the ordinary differential equation \eqref{eq:equation-of-W-tilde}-\eqref{bdy-cdts-of-W}. 
Thanks to \eqref{def:tildeV}, \eqref{def:tildeW} and \eqref{def:W-gamma-z}, we have in particular that
\begin{equation}\label{def:V-in-terms-of-W}
\begin{aligned}
    V(\tilde z ) 
    &= 
    -  
    \Big( \tau + \frac{U_s''(a)}{2} \tilde z^2 \Big) H(\tilde z) + 
    \Big( \tau + \frac{U_s''(a)}{2} \tilde z^2 \Big) \tilde W (\tilde z)
    \\
    &=
    \bigg(
        \tau   
        + \frac{U_s''(a)}{2} \tilde z^2 
    \bigg) 
    \bigg(
    W 
    \bigg( 
       \sqrt[3]{\frac{|U_s''(a)|}{2} }
       \tilde z
    \bigg)
        -
        H(\tilde z)
    \bigg),
    \qquad 
    \text{for all} \quad \tilde z \in \mathbb R.
\end{aligned}
\end{equation}
As predicted in \Cref{rmk:V-has-a-jump}, we observe that $V$ has a jump in $\tilde z = 0$ ($W$ is smooth in $\mathbb R$) and its derivative behaves as the Dirac delta $-\tau \delta(\tilde z)$ near the origin. Furthermore, $V$ and its derivatives decay exponentially as $z\to \pm \infty$, since from \eqref{bdy-cdts-of-W}, \eqref{eq:decay-of-X} and \eqref{eq:relation-between-W-and-X} we have that 
\begin{equation}\label{exponential-decay-of-W}
\begin{alignedat}{4}
    \big| 
        W (z)-1
    \big|
    &\leq 
    \frac{
        \int_{z}^\infty
        | X(s) |ds
    }
    {\int_{\mathbb R}
        | X(s) |ds
    }
    &&
    \leq 
    C_W
    \exp \bigg\{ - \frac{\sqrt{\av{\gamma}} |  z|^2}{4} \bigg\}
    \qquad 
    z\geq 0\\
    \big| 
        W 
        (z)
    \big|
    &\leq 
    \frac{
        \int_{-\infty}^z
        | X(s) |ds
    }
    {\int_{\mathbb R}
        | X(s) |ds
    }
    &&
    \leq 
    C_W
    \exp \bigg\{ - \frac{\sqrt{\av{\gamma}}|  z|^2}{4} \bigg\}
    \qquad 
    z<0,\\
    \big| 
        W' 
        (z)
    \big|
    &\leq 
    \frac{
        | X(z) |
    }
    {\int_{\mathbb R}
        | X(s) |ds
    }
    &&\leq 
    C_W
    \exp \bigg\{ - \frac{\sqrt{\av{\gamma}}|  z|^2}{4} \bigg\}
    \qquad 
    z\in \mathbb R,
\end{alignedat}
\end{equation}
for an harmless constant $C_W>0$, that depends only upon the function $W$. 
In particular, the profile $\mathbb V_\ee$ in \eqref{def-v-ee} can be written in terms of $W$ and $\tau$, by means of
\begin{equation}\label{def-mathbbVee}
    \boxed{
    \begin{aligned}
        &\mathbb{V}_\ee (y)
        =
        \Big(
            U_s(y) +\ee^{\frac{2}{3}}\tau 
        \Big)
        H(y-a)
        + 
        \ee^\frac{2}{3} 
        V\left(\frac{y-a}{\ee^{1/3}} \right) \\
        &=\!
        U_s(y)
        H(y\!-\!a) 
        \!+\!
        \ee^\frac{2}{3}
        \tau 
        W 
        \bigg( \!\!
          \sqrt[3]{\frac{|U_s''(a)|}{2\ee}}
         (y\!-\!a)\!
        \bigg)
        \!+ \!
        \frac{U_s''(a) (y-a)^2}{2}
        \bigg\{\!
        W 
        \bigg(\! \!
          \sqrt[3]{\frac{|U_s''(a)|}{2\ee}}
         (y\!-\!a)\!
        \bigg)
        \!-\!
        H(y-a)
        \!\bigg\}
        \!.
    \end{aligned}
     }
\end{equation}
Next, we determine the profile $\mathbb U_k$ of \Cref{prop:main-prop}, recalling that $\ee = \ee(k) = 1/k$ and $\mathbb U_k =\mathbb  U_\ee = i\mathbb V_{\ee}'(y)$:
\begin{equation}\label{eq:final-form-of-Uk}
\boxed{
    \begin{aligned}
        \mathbb{U}_k (y)
        =
        i U_s'(y)
        H(y-a)
        +
        i\tau  
        k^{-\frac{1}{3}}
        \sqrt[3]{\frac{|U_s''(a)|}{2}}
        W' 
        \bigg( 
         \sqrt[3]{\frac{|U_s''(a)|}{2}}
         k^\frac{1}{3}
         (y-a)
        \bigg)
        + \\
        +
        i
        U_s''(a)(y-a)
        \bigg\{
        W 
        \bigg(
          \sqrt[3]{\frac{|U_s''(a)|}{2}}
          k^\frac{1}{3}
         (y-a)
        \bigg)
        -
        H(y-a)
        \bigg\}
        + 
        \\
        +
        i
        k^\frac{1}{3}
        \sqrt[3]{\frac{|U_s''(a)|}{2}}
        \frac{U_s''(a) (y-a)^2}{2}
        W' 
        \bigg(
          \sqrt[3]{\frac{|U_s''(a)|}{2}}
          k^\frac{1}{3}
         (y-a)
        \bigg).
    \end{aligned}
     }
\end{equation}
The profile $\mathbb U_k$ is in $W^{2, \infty}_\alpha(\mathbb R_+)$ for all $k\in \mathbb N$. Furthermore it is uniformly bounded from below and above in $W^{0, \infty}_\alpha(\mathbb R_+)$, since
\begin{align*}
      \bigg\| 
        i\tau  
        k^{-\frac{1}{3}}
        \sqrt[3]{\frac{|U_s''(a)|}{2}}
        W' 
        \bigg( 
         \underbrace{
         \sqrt[3]{\frac{|U_s''(a)|}{2}}
         k^\frac{1}{3}
         (y-a)
         }_{=:z}
        \bigg)
    \bigg\|_{W^{0, \infty}_\alpha}
    \leq 
    C_We^{\alpha a}
    \frac{ |\tau |}{\sqrt[3]{k}}
    \sqrt[3]{\frac{|U_s''(a)|}{2}}
    \sup_{z\in \mathbb R}
    \exp \bigg\{ \frac{\alpha |z|}{\sqrt[3]{k}} - \frac{\sqrt{\av{\gamma}}|  z|^2}{4} \bigg\},
\end{align*}
as well as
\begin{align*}
      \bigg\| 
        i
        U_s''(a)(y-a)
        \bigg\{
        W 
        \bigg(
          \sqrt[3]{\frac{|U_s''(a)|}{2}}
          k^\frac{1}{3}
         (y-a)
        \bigg)
        -
        H(y-a)
        \bigg\}
    \bigg\|_{W^{0, \infty}_\alpha}\\
    \leq 
    C_W e^{\alpha a}
    \frac{ 1
    }{\sqrt[3]{k}}
    \sqrt[3]{\frac{|U_s''(a)|}{2}}
    \sup_{z\in \mathbb R}
    |z|
    \exp \bigg\{ \frac{\alpha |z|}{\sqrt[3]{k}} - \frac{\sqrt{\av{\gamma}}|  z|^2}{4} \bigg\},
\end{align*}
and finally
\begin{align*}
    \bigg\|
    i
    &   k^\frac{1}{3}
        \sqrt[3]{\frac{|U_s''(a)|}{2}}
        \frac{U_s''(a) (y-a)^2}{2}
        W' 
        \bigg(
          \sqrt[3]{\frac{|U_s''(a)|}{2}}
          k^\frac{1}{3}
         (y-a)
        \bigg)
    \bigg\|_{W^{0, \infty}_\alpha}\\
    &\leq 
    \frac{C_We^{a\alpha}}{\sqrt[3]{k}}    
    \bigg(
        \frac{|U_s''(a)|}{2}
    \bigg)^\frac{2}{3}
    \sup_{z \in \mathbb R}
    z^2\exp \bigg\{ \frac{\alpha |z|}{\sqrt[3]{k}} - \frac{\sqrt{\av{\gamma}}|  z|^2}{4} \bigg\}.
\end{align*}
These values converge all towards $0$ as $k\to \infty$, thus $\mathbb U_k$ converges towards $iU_s'(y)H(y-a)$ in $W^{0,\infty}_\alpha$. 

\noindent 
We now take the difference between equation \eqref{eq:final-equation-for-Veps-with-rest} for a positive $\ee>0$ and \eqref{eq:V-distribution}:
\begin{equation*}
\begin{aligned}
    \ee^{-1}
    \mathcal{R}_\ee(a+\sqrt[3]{\ee} \tilde z)
    = 
    i\tau 
    \bigg\{
        \Big(
            \ee^{-\frac{2}{3}}U_s(a+\sqrt[3]{\ee} \tilde z)
            -
            \frac{U_s''(a)}{2}
            \tilde z^2
        \Big)
        V'(\tilde z)
        -
        \Big(
            \ee^{-\frac{1}{3}}U_s'(a+\sqrt[3]{\ee} \tilde z)
            -
            U_s''(a)
            \tilde z
        \Big)
        V(\tilde z)
    \bigg\}
    + \\
    +
    \ee^\frac{1}{3}
    \bigg\{
        \tau^2 \delta (\tilde z) + 
        \tau V'(\tilde z) +
        \ee^{-\frac{2}{3}}
        U_s(a+\sqrt[3]{\ee} \tilde z)
        V'(\tilde z)
        -
        \ee^{-\frac{1}{3}}
        U_s'(a+\sqrt[3]{\ee} \tilde z)
        V(\tilde z)
    \bigg\}
    +
    i \ee^\frac{1}{3} H(\tilde z) U_s'''(a+\sqrt[3]{\ee} \tilde z).
\end{aligned}
\end{equation*}
Hence, replacing the formula \eqref{def:V-in-terms-of-W} of $V$ and recasting the equation in terms of $y = a + \sqrt[3]{\ee}\tilde z>0$, we finally deduce that
\begin{equation}\label{eq:explicit-Rk}
\boxed{
\begin{aligned}
    \mathcal{R}_k(y)
    = 
    i 
    k^{-\frac{4}{3}}
    H(y-a) U_s'''(y) 
    +
    \bigg(
        W
        \bigg(
          \sqrt[3]{\frac{|U_s''(a)|}{2}}
          k^\frac{1}{3}
         (y-a)
        \bigg)
        -
        H(y-a)
    \bigg)
    \mathcal{F}_{1,k}(y)
    +\\
    +
    W'
    \bigg(
       \sqrt[3]{\frac{|U_s''(a)|}{2}}
       k^\frac{1}{3}
       (y-a)
    \bigg)
    \mathcal{F}_{2,k}(y),
\end{aligned}
}
\end{equation}
where the functions $\mathcal{F}_{1,k}$ and $\mathcal{F}_{2,k}$ stand for 
\begin{align*}
    \mathcal{F}_{1,k}(y) 
    =
    i\tau U_s''(a)(y-a)
    \Big(
        U_s(y) - \frac{U_s''(a)}{2}(y-a)^2
    \Big)
    -
    i\tau 
    \Big(
        U_s'(y) - U_s''(a)(y-a)
    \Big)
    \Big(
        k^{-\frac{2}{3}}\tau 
        +
        \frac{U_s''(a)}{2}
        (y-a)^2
    \Big)
    +
    \\
    +
    k^{-1}\tau U_s''(a)(y-a) + 
    k^{-\frac{1}{3}}U_s(y)U_s''(a)(y-a) -
    k^{-\frac{1}{3}} U_s'(y) 
    \Big(
        k^{-\frac{2}{3}}\tau 
        +
        \frac{U_s''(a)}{2}
        (y-a)^2
    \Big)
    \end{align*}
    and
\begin{align*}
    \mathcal{F}_{2,k}(y) 
    = 
    \pare{\frac{|U_s''(a)|}{2}}^{1/3}
    \pare{
        k^{-\frac{2}{3}}\tau 
        +
        \frac{U_s''(a)}{2}
        (y-a)^2
    }
    \bigg(
    i
    \tau 
    k^{-\frac{2}{3}}
    +
    k^{-1}
    +
    k^{-\frac{1}{3}}
    U_s(y)
    \bigg).
\end{align*}
With a similar procedure as the one used for $\mathbb U_k$, we infer that also $(k^\frac{4}{3}\mathcal{R}_k)_{k\in \mathbb N}$ is uniformly bounded in $W^{0, \infty}_\alpha$. To this end, we shall observe that the multiplication of $k^\frac{4}{3}$ with $\mathcal{F}_{1,k}$ and $\mathcal{F}_{2,k}$ can be bounded in $y\in \mathbb R_+$ by means of polynoms depending on $\sqrt[3]{k}|y-a|$. For instance, applying the Taylor formula to the first term of $k^\frac{4}{3}\mathcal{F}_{1,k}$
\begin{equation*}
    \Big|
    k^\frac{4}{3}i\tau U_s''(a)(y-a)
    \Big(
        U_s(y) - \frac{U_s''(a)}{2}(y-a)^2
    \Big)
    \Big|
    \leq 
    |\tau |
    |U_s''(a)| 
    \sqrt[3]{k}
    |y-a|
    \| U_s''' \|_{W^{0, \infty}_\alpha}
    (\sqrt[3]{k}
    |y-a|)^3.
\end{equation*}
Similar estimates hold true for the remaining terms of $k^\frac{4}{3}\mathcal{F}_{1,k}$ and of 
$k^\frac{4}{3}\mathcal{F}_{2,k}$. This fact together with the exponential decay of \eqref{exponential-decay-of-W} (which in \eqref{eq:explicit-Rk} depend on $\sqrt[3]{k}|y-a|$), ensure that  $(\mathcal R_k)_{k\in \mathbb N}$ is uniformly bounded in $W^{0, \infty}_\alpha$. This conclude the proof of \Cref{prop:main-prop}. 

\begin{appendix}

\section{Proof of \Cref{thm:refined-verion-GV-Dormy}}\label{sec:proof-of-GV-theorem}
In this section we outline the proof of \Cref{thm:refined-verion-GV-Dormy}. Similarly to the procedure introduced in \Cref{sec:outline-of-the-proof}, we address at any frequency $k\in \mathbb N$ the following ``forced'' Prandtl equation 
\begin{equation}
    \label{Prandtl-k-forced}
    \system{
    \begin{alignedat}{4}
    & 
    \partial_t u_k^{\rm fr}  + ik U_{{\rm s}} u_k^{\rm fr} +  k v_k^{\rm fr} U'_{{\rm s}} - \partial_{y}^2 u_k^{\rm fr} =f_k,
    \qquad (t,y)\in \,
    &&(0,T) \times \mathbb R_+,\\
    &\, i u_k^{\rm fr} + \partial_y v_k^{\rm fr} =0
    &&(0,T) \times \mathbb R_+,
    \\
    & u_k^{\rm fr}|_{t=0} = u_{k,\rm in} 
    &&\hspace{1.425cm}  \mathbb R_+,
    \\
    & \left. u_k\right|_{y=0}=0,\quad  \left. v_k \right|_{y = 0} = 0
     &&(0,T),
    \end{alignedat}
    } 
\end{equation}
and we denote by $T_k(t):W^{0, \infty}_\alpha \to W^{0, \infty}_\alpha$, for any $t\in (0,T)$, the semigroup generated by the homogeneous system ($f_k\equiv 0$). Our starting point is the following proposition, that somehow recasts the main result of G\'erard-Varet in \cite{GD2010} in terms of solutions of \eqref{Prandtl-k-forced}.
\begin{prop}\label{prop:main-prop-Prandtl}
Let $U_s \in W^{4,\infty}_\alpha(\mathbb R_+)$, $U_s'(a) = 0$ and $U_s''(a) \neq  0$, for a given $a\in (0, \infty)$. There exists a complex number $\tau \in \mathbb C$  with ${\mathcal{I}m(\tau)}<0$, such that, for any $k\in \mathbb N$, there exist $\mathbb U_k \in W^{1, \infty}_\alpha (\mathbb R_+)$  and $\mathcal{R}_k \in W^{0, \infty}_\alpha (\mathbb R_+)$ so that
\begin{equation}\label{ansatz:main-proposition-Prandtl}
     u_{k,\rm in}(y) =  \mathbb{U}_k(y),\quad 
     f_k(t,y)  = -k e^{i\tau k^{\frac{1}{2}}t}\mathcal{R}_k(y),
     \quad 
     (t,y)\in (0,T)\times \mathbb R_+,
\end{equation}
generate a global-in-time solution $u_k^{\rm fr}\in L^\infty(0,T;W^{0, \infty}_\alpha)$ of \eqref{Prandtl-k-forced}, which can be  written explicitly as
\begin{equation}\label{eq:explicit-form-ukfr-Prandtl}
\begin{aligned}
    u_k^{\rm fr}(t,y) &= e^{i \tau k^{\frac{1}{2}}t} \mathbb{U}_k(y),\quad 
    v_k^{\rm fr}(t,y) = k e^{i \tau k^{\frac{1}{2}}t} \mathbb{V}_k(y),
    \quad\text{with}\quad  
    \mathbb{V}_k(y) :=-i \int_0^y\mathbb{U}_k(z)dz.
\end{aligned}
\end{equation}
Moreover there exist three constants $\smallc,\mathcal{C},\mathcal{C}_\mathcal{R}>0$ such that 
for any  $k\in \mathbb N$
\begin{equation}\label{uniform-estimate-for-Uk-Prandtl}
0<\smallc \leq 
\| \mathbb{U}_k\|_{W^{0, \infty}_\alpha(\mathbb R_+)}
\leq \mathcal{C}<\infty,\qquad 
k \| \mathcal{R}_k \|_{W^{0, \infty}_\alpha} 
    \leq 
    \mathcal{C}_\mathcal{R}.
\end{equation} 
\end{prop}
\noindent 
The proof of \Cref{prop:main-prop-Prandtl} is similar to the one of \Cref{prop:main-prop}. In essence, it was performed in \cite{GD2010}, from Section 2 to Section 4.1 (although formulated differently). We shall here report uniquely the exact form of $\mathbb U_k$ and $\mathcal{R}_k$:
\begin{equation}\label{def-Uk-Prandtl}
    \boxed{
    \begin{aligned}
        \mathbb{U}_k (y)
        =
        i U_s'(y) 
        H(y-a) 
        +
        \frac{i}{\sqrt[4]{k}}
        \bigg(
        \tau 
        +
        \sqrt[4]{\frac{|U_s''(a)|}{2}}
        \frac{U_s''(a) (\sqrt[4]{k}(y-a))^2}{2}
        \bigg)
        W '
        \bigg( 
          \sqrt[4]{\frac{|U_s''(a)|}{2}}
          k^\frac{1}{4}
         (y-a)
        \bigg)
        +
        \\
        +
        i 
        U_s''(a) (y-a)
        \bigg\{ 
        W 
        \bigg( 
          \sqrt[4]{\frac{|U_s''(a)|}{2}}
         k^\frac{1}{4}(y - a) 
        \bigg)
         - 
        H(y-a)
         \bigg\}
        .
    \end{aligned}
     }
\end{equation} 
and
\begin{equation}\label{eq:explicit-Rk-Prandtl}
\boxed{
\begin{aligned}
    \mathcal{R}_k(y)
    = 
    i 
    k^{-1}
    H(y-a) U_s'''(y) 
    +
    \mathcal{F}_{1,k}(y)
    \Bigg\{
        W
        \bigg(
            &
            \sqrt[4]{\frac{|U_s''(a)|}{2} }
          k^\frac{1}{4}
         (y-a)
        \bigg)
        -
        H(y-a)
    \Bigg\}
    +\\
    &+
    \mathcal{F}_{2,k}(y)
    W'
    \bigg(
       \sqrt[4]{\frac{|U_s''(a)|}{2} }
       k^\frac{1}{4}
       (y-a)
    \bigg)
    ,
\end{aligned}
}
\end{equation}
where the functions $\mathcal{F}_{1,k}$ and $\mathcal{F}_{2,k}$ stand for 
\begin{equation}\label{F1Prandtl}
    \begin{aligned}
    \mathcal{F}_{1,k}(y) 
    &=
    \bigg(
            U_s(y)
            -
            U_s(a)
            -
            \frac{U_s''(a)}{2}
                (y-a)^2
    \bigg)
     U_s''(a)(y-a) 
     +
     \\
     &\hspace{3cm}-
     \Big(
       U_s'(y)
        -
       U_s''(a)
       (y-a)
    \Big)
    \bigg(
        \tau k^{-\frac 1 2}  
        + \frac{U_s''(a)}{2} 
            (y-a)^2 
    \bigg)
    \end{aligned}
\end{equation}
    and
\begin{equation}\label{F2Prandtl}
    \begin{aligned}
    \mathcal{F}_{2,k}(y) 
    =
    \bigg(
            U_s(y)
            -
            U_s(a)
            -
            \frac{U_s''(a)}{2}
            (y-a)^2
    \bigg)
    \bigg(
          \tau k^{-\frac 12}   
          + \frac{U_s''(a)}{2}(y-a)^2 
    \bigg).
    \end{aligned}
\end{equation}
We next set $\mathcal{C}_\sigma $ in \Cref{thm:refined-verion-GV-Dormy} as
\begin{equation*}
    \mathcal{C}_\sigma 
    :=
    \frac{\smallc}{2}
        \frac{ (\sigma_0-\sigma)
        }{(\sigma_0-\sigma) + \mathcal{C}_R}.
\end{equation*}
Assume by contradiction that there exists a frequency $k  \in \mathbb N$, so that
    \begin{equation}\label{the-contradiction-Prandtl}
        \sup_{0\leq t \leq \frac{1}{2(\sigma_0-\sigma)}
        \frac{\ln(k)}{\sqrt {k}}
        }
        e^{-\sigma t k^\frac{1}{2}}
        \big\|  T_k(t) 
        \big\|_{\mathcal{L}(  W^{0, \infty}_\alpha)}
        \leq 
        \frac{\smallc}{2}
        \frac{ (\sigma_0-\sigma)
        }{(\sigma_0-\sigma) + \mathcal{C}_R}
        k^\frac{1}{2},
    \end{equation}
    We consider the global-in-time solution $u_k^{\rm fr}(t,y) = e^{i \tau k^{1/2}t} \mathbb{U}_k(y)$ provided by \Cref{prop:main-prop-Prandtl}, which by uniqueness also satisfies the Duhamel's relation
    \begin{equation*}
        u_k^{\rm fr}(t,y)
        = 
        T_k(t)\big(\mathbb{U}_k \big)(y)
        + 
        \int_0^t
        T_k(t-s)(  -k e^{i\tau k^{\frac{1}{2}}s}\mathcal{R}_k(s))(y)ds,
    \end{equation*}
    Hence, applying the $W^{0, \infty}_\alpha$-norm to this identity and applying the triangular inequality, we gather that
    \begin{align*}
        \|  T_k(t)
            &\big(\mathbb{U}_k \big) \|_{W^{0, \infty}_\alpha}
        \geq 
        \| u_k^{\rm fr}(t) \|_{W^{0, \infty}_\alpha } - 
        \int_0^t 
        \| 
            T_k(t-s)(  -k e^{i\tau k^{\frac{1}{2}}s}
            \mathcal{R}_k(s)) 
        \|_{W^{0, \infty}_\alpha }
        ds
        \\
        &
        \geq
        \smallc 
        e^{\sigma_0 k^{\frac{1}{2}}t}
        - 
        \int_0^t 
        \| T_k (t-s) \|_{\mathcal{L}}
        k 
        \| \mathcal{R}_k(s)
        \|_{W^{0, \infty}_\alpha }
        e^{\sigma_0 s k^\frac{1}{2}}
        ds
        \\
        &
        \geq
        \smallc 
        e^{\sigma_0 k^{\frac{1}{2}}t}
        - 
        \int_0^t 
        \| T_k (t-s) \|_{\mathcal{L}}
        e^{-\sigma k^\frac{1}{2}(t-s) }
        k 
        \| \mathcal{R}_k(s)
        \|_{W^{0, \infty}_\alpha }
        e^{-(\sigma_0-\sigma)(t- s)k^\frac{1}{2}}
        ds
        e^{\sigma_0 t k^\frac{1}{2}},
    \end{align*}
    where we have estimated $\| u_k^{\rm fr}(t) \|_{W^{0, \infty}_\alpha }$ with \eqref{uniform-estimate-for-Uk-Prandtl}. Hence applying \eqref{the-contradiction-Prandtl} and invoking the uniform bound  on   $\mathcal{R}_k$, we obtain
    \begin{align*}
        \|  T_k(t)
            &\big(\mathbb{U}_k \big) \|_{W^{0, \infty}_\alpha}
        \geq
        \smallc 
        e^{\sigma_0 k^{\frac{1}{2}}t}
        - 
        \frac{\smallc }{2}
        \frac{(\sigma_0-\sigma)}{(\sigma_0-\sigma) + \mathcal{C}_R}
        k^\frac{1}{2}
        \mathcal{C}_R
        \int_0^t 
        e^{-(\sigma_0-\sigma)(t-s) k^\frac{1}{2}}
        ds
        e^{\sigma_0 k^\frac{1}{2} t}.
    \end{align*}
    Multiplying both l.~and r.h.s.~by $e^{-\sigma_0 k^{1/2}t }/\small c$ and calculating explicitly the integral on the r.h.s, we obtain
    \begin{align*}
        \frac{e^{-\sigma_0 k^\frac{1}{2}t }}{\small c}
        \|  T_k(t)
            &\big(\mathbb{U}_k \big) \|_{W^{0, \infty}_\alpha}
        \geq
        1
        - 
        \frac{1}{2}
        \frac{\mathcal{C}_R}{(\sigma_0-\sigma) + \mathcal{C}_R}
        \left(
                1
                -
                e^{-(\sigma_0-\sigma)k^\frac{1}{2}t}
        \right).
    \end{align*}
    Recasting $e^{-\sigma_0 k^\frac{1}{2}t } = e^{-(\sigma_0-\sigma) k^\frac{1}{2}t }e^{-\sigma k^\frac{1}{2}t } $ on the l.h.s., we apply once more \eqref{the-contradiction-Prandtl}, to deduce that 
    \begin{align*}
        \frac{1}{2}
        \frac{ (\sigma_0-\sigma)
        }{(\sigma_0-\sigma) + \mathcal{C}_R}
        e^{-(\sigma_0-\sigma) k^\frac{1}{2}t }
        k^\frac{1}{2}
        \geq
        1
        - 
        \frac{1}{2}
        \frac{\mathcal{C}_R}{(\sigma_0-\sigma) + \mathcal{C}_R}
        \left(
                1
                -
                e^{-(\sigma_0-\sigma)k^\frac{1}{2}t}
        \right),
    \end{align*}
    for any time $t \in \big[0, \frac{1}{2(\sigma_0-\sigma)}\frac{\ln(k)}{\sqrt{k}}\big]$. 
    We hence set $ t =\frac{1}{2(\sigma_0-\sigma)}
    \frac{\ln(k)}{\sqrt{k}} $, which in particular implies that
    $e^{-(\sigma_0-\sigma)k^{1/2}t} = k^{-1/2}$, so that    
     \begin{align*}
        \frac{1}{2}
        \frac{ (\sigma_0-\sigma)}{(\sigma_0-\sigma) + \mathcal{C}_R}
        \geq
        1
        - 
        \frac{1}{2}
        \frac{\mathcal{C}_R(
                1
                -
               k^{-\frac{1}{2}})}{(\sigma_0-\sigma) + \mathcal{C}_R}
        .
    \end{align*}
    By bringing the last term on the r.h.s~to the l.h.s., we finally obtain that 
    \begin{equation*}
        \frac{1}{2}
        =
        \frac{1}{2}
        \frac{ (\sigma_0-\sigma)+\mathcal{C}_R}{(\sigma_0-\sigma) + \mathcal{C}_R}
        \geq 
        \frac{1}{2}
        \frac{ (\sigma_0-\sigma)+\mathcal{C}_R (
                1
                -
               k^{-\frac{1}{2}})}{(\sigma_0-\sigma) + \mathcal{C}_R}
               \geq 1,
    \end{equation*}
    which is a contradiction. This concludes the proof of \Cref{thm:refined-verion-GV-Dormy}.
    \hfill $\Box$ \\

  \begin{remark}\label{remark:it-might-be-true}
  Let us notice that if we allow both the shear flow $U_s$ and the solutions of the Prandtl equations to belong to a wider space, namely $W^{1, \infty}_{\textnormal{loc}}\pare{\mathbb R_+}$, we can obtain an explicit solution of Prandtl, which behaves as $e^{\sigma_0\sqrt{k}t}$, at any time $t>0$, but does not decay as $y\to \infty$. By setting indeed   
  \begin{equation*}
      U_s(y):= \frac{(y-a)^2}{2},
  \end{equation*}
  we remark that the forcing term $\mathcal{R}_k$ in \eqref{eq:explicit-Rk-Prandtl} vanishes, thus $ u(t,x,y) := e^{i\tau \sqrt{k}t+ ikx}  \mathbb U_k(y)$ (with $\mathbb U_k$ as in \eqref{def-Uk-Prandtl}) is an exact solution of the Prandtl equations. However $\mathbb U_k(y) \sim i U_s'(y) = (y-a)$ as $y\gg 1$, namely the profile $\mathbb U_k$ behaves similarly as a couette flow, which does not decay as $y\to \infty$.      
  \end{remark}
  



\end{appendix}

\section*{Acknowledgment}

The third author was partially supported by GNAMPA and INDAM. The authors are thankful to Prof.~D.~G\'erard-Varet for constructive remarks and comments that helped to improve the manuscript.

\subsection*{Data Availability Statement} 
Data sharing is not applicable to this article, since no datasets were generated or analysed during the current study.

\subsection*{Conflict of interest} 
The authors declare that they have no conflict of interest. 

\printbibliography


\end{document}